\documentclass[11pt]{article}

\usepackage{xcolor}
\usepackage{amsmath}
\usepackage{amsfonts}
\usepackage{amssymb}
\usepackage{subfigure}
\usepackage{epsf,fancybox}
\usepackage{mathrsfs}
\usepackage{color}
\usepackage{multirow}
\usepackage{paralist}
\usepackage{verbatim}
\usepackage{galois}
\usepackage{algorithm}
\usepackage{algorithmic}
\usepackage{boxedminipage}
\usepackage{booktabs}
\usepackage{accents}
\usepackage{stmaryrd}

\usepackage{natbib}

\usepackage{url}
\usepackage[colorlinks,linkcolor=magenta,citecolor=blue, pagebackref=true,backref=true]{hyperref}
\renewcommand*{\backrefalt}[4]{%
    \ifcase #1 \footnotesize{(Not cited.)}%
    \or        \footnotesize{(Cited on page~#2.)}%
    \else      \footnotesize{(Cited on pages~#2.)}%
    \fi}

\newtheorem{theorem}{Theorem}[section]

\newtheorem{lemma}[theorem]{Lemma}

\newtheorem{definition}[theorem]{Definition}

\newtheorem{remark}[theorem]{Remark}
\newtheorem{assumption}[theorem]{Assumption}

\newcommand{\proj}{\mathcal{P}}

\newcommand{\op}{\textnormal{op}}

\newcommand{\x}{\mathbf x}

\newcommand{\argmin}{\mathop{\rm{argmin}}}
\newcommand{\argmax}{\mathop{\rm{argmax}}}

\newcommand{\ECal}{\mathcal{E}}

\newcommand{\XCal}{\mathcal{X}}
\newcommand{\YCal}{\mathcal{Y}}
\newcommand{\WCal}{\mathcal{W}}

\newcommand{\dist}{\textbf{dist}}

\newcommand{\br}{\mathbb{R}}

\newcommand{\ba}{\begin{array}}
\newcommand{\ea}{\end{array}}

\newcommand{\ZCal}{\mathcal{Z}}

\begin{document}


\begin{center}

{\bf{\LARGE{Perseus: A Simple and Optimal High-Order Method for \\ [.2cm] Variational Inequalities}}}

\vspace*{.2in}
{\large{ \begin{tabular}{c}
Tianyi Lin$^\ddagger$ \and Michael I. Jordan$^{\diamond, \dagger}$ \\
\end{tabular}
}}

\vspace*{.2in}

\begin{tabular}{c}
Department of Electrical Engineering and Computer Sciences$^\diamond$ \\
Department of Statistics$^\dagger$ \\
University of California, Berkeley \\
Laboratory for Information and Decision Systems (LIDS), MIT$^\ddagger$
\end{tabular}

\vspace*{.2in}

\today

\vspace*{.2in}

\begin{abstract}
This paper settles an open and challenging question pertaining to the design of simple and optimal high-order methods for solving smooth and monotone variational inequalities (VIs).  A VI involves finding $x^\star \in \XCal$ such that $\langle F(x), x - x^\star\rangle \geq 0$ for all $x \in \XCal$.  We consider the setting in which $F: \br^d \rightarrow \br^d$ is smooth with up to $(p-1)^{\textnormal{th}}$-order derivatives. For $p = 2$, the cubic regularization of Newton's method has been extended to VIs with a global rate of $O(\epsilon^{-1})$~\citep{Nesterov-2006-Constrained}. An improved rate of $O(\epsilon^{-2/3}\log\log(1/\epsilon))$ can be obtained via an alternative second-order method, but this method requires a nontrivial line-search procedure as an inner loop. Similarly, the existing high-order methods based on line-search procedures have been shown to achieve a rate of $O(\epsilon^{-2/(p+1)}\log\log(1/\epsilon))$~\citep{Bullins-2022-Higher,Lin-2023-Monotone,Jiang-2022-Generalized}. As emphasized by~\citet{Nesterov-2018-Lectures}, however, such procedures do not necessarily imply the practical applicability in large-scale applications, and it is desirable to complement these results with a simple high-order VI method that retains the optimality of the more complex methods. We propose a $p^{\textnormal{th}}$-order method that does \textit{not} require any line search procedure and provably converges to a weak solution at a rate of $O(\epsilon^{-2/(p+1)})$. We prove that our $p^{\textnormal{th}}$-order method is optimal in the monotone setting by establishing a lower bound of $\Omega(\epsilon^{-2/(p+1)})$ under a generalized linear span assumption. A restarted version of our $p^{\textnormal{th}}$-order method attains a linear rate for smooth and $p^{\textnormal{th}}$-order uniformly monotone VIs and another restarted version of our $p^{\textnormal{th}}$-order method attains a local superlinear rate for smooth and strongly monotone VIs. Further, the similar $p^{\textnormal{th}}$-order method achieves a global rate of $O(\epsilon^{-2/p})$ for solving smooth and nonmonotone VIs satisfying the Minty condition. Two restarted versions attain a global linear rate under additional $p^{\textnormal{th}}$-order uniform Minty condition and a local superlinear rate under additional strong Minty condition. 
\end{abstract}

\end{center}

\section{Introduction}\label{sec:intro}
Let $\br^d$ be a finite-dimensional Euclidean space and let $\XCal \subseteq \br^d$ be a closed, convex and bounded set with a diameter $D > 0$. Given that $F: \br^d \rightarrow \br^d$ is a continuous operator, a fundamental assumption in optimization theory, generalizing convexity, is that $F$ is \textit{monotone}:
\begin{equation*}
\langle F(x) - F(x'), x- x'\rangle \geq 0, \quad \textnormal{for all } x, x' \in \br^d.  
\end{equation*}
Another useful assumption in this context is that $F$ is $(p-1)^{\textnormal{th}}$-order $L$-smooth; in particular, that it has Lipschitz-continuous $(p-1)^{\textnormal{th}}$-order derivative ($p \geq 1$) in the sense that there exists a constant $L > 0$ such that
\begin{equation}\label{eq:smooth}
\|\nabla^{(p-1)} F(x) -  \nabla^{(p-1)} F(x')\|_\op \leq L\|x - x'\|, \quad \textnormal{for all } x, x' \in \br^d.  
\end{equation}
With these assumptions, we can formulate the main problem of interest in this paper---the \textit{Minty variational inequality} problem~\citep{Minty-1962-Monotone}.  This  consists in finding a point $x^\star \in \XCal$ such that
\begin{equation}\label{prob:Minty}
\langle F(x), x- x^\star\rangle \geq 0, \quad \textnormal{for all } x \in \XCal.  
\end{equation}
The solution to Eq.~\eqref{prob:Minty} is often referred to as a \textit{weak} solution to the variational inequality (VI) corresponding to $F$ and $\XCal$~\citep{Facchinei-2007-Finite}. By way of comparison, the \textit{Stampacchia variational inequality} problem~\citep{Hartman-1966-Some} consists in finding a point $x^\star \in \XCal$ such that
\begin{equation}\label{prob:Stampacchia}
\langle F(x^\star), x- x^\star\rangle \geq 0, \quad \textnormal{for all } x \in \XCal,
\end{equation}
and the solution to Eq.~\eqref{prob:Stampacchia} is called a \textit{strong} solution to the VI corresponding to $F$ and $\XCal$. In the setting where $F$ is continuous and monotone, the solution sets of Eq.~\eqref{prob:Minty} and Eq.~\eqref{prob:Stampacchia} are equivalent. However, these two solution sets are different in general and a weak solution need not exist when a strong solution exists. In addition, computing an approximate strong solution involves a higher computational burden than finding an approximate weak solution~\citep{Monteiro-2010-Complexity,Monteiro-2011-Complexity,Chen-2017-Accelerated}. Earlier works have focused on the asymptotic global convergence analysis of various VI methods under mild conditions~\citep{Lemke-1964-Equilibrium,Scarf-1967-Approximation,Todd-1976-Computation,Hammond-1987-Generalized,Fukushima-1992-Equivalent,Magnanti-1997-Orthogonality}. Two notable exceptions are the generalizations of the ellipsoid method~\citep{Magnanti-1995-Unifying} and the interior-point method~\citep{Ralph-1997-Superlinear}, both of which have been the subject of nonasymptotic complexity analysis.

VIs capture a wide range of problems in optimization theory and beyond, including saddle-point problems and models of equilibria in game-theoretic settings~\citep{Cottle-1980-Variational,Kinderlehrer-2000-Introduction,Tremolieres-2011-Numerical}. Moreover, the challenge of designing solution methods for VIs with provable worst-case bounds has driven significant research over several decades; see~\citep[e.g.,][]{Harker-1990-Finite,Facchinei-2007-Finite}. This research has provided a foundation for work in machine learning in recent years, where general saddle-point problems have emerged in many settings, including generative adversarial networks (GANs)~\citep{Goodfellow-2014-Generative} and multi-agent learning in games~\citep{Cesa-2006-Prediction,Mertikopoulos-2019-Learning}. Some of these applications in ML induce a nonmonotone structure, with representative examples including the training of robust neural networks~\citep{Madry-2018-Towards} or robust classifiers~\citep{Sinha-2018-Certifiable}. 

Building on seminal work in the context of high-order optimization~\citep{Baes-2009-Estimate,Birgin-2017-Worst}, we tackle the challenge of developing $p^{\textnormal{th}}$-order methods for VIs via an inexact solution of regularized subproblems obtained from a $(p-1)^{\textnormal{th}}$-order Taylor expansion of $F$. Accordingly, we make the following assumptions throughout this paper.
\begin{itemize}
\item[\textbf{A1}.] $F: \br^d \rightarrow \br^d$ is $(p-1)^{\textnormal{th}}$-order $L$-smooth.  
\item[\textbf{A2}.] The subproblem based on a $(p-1)^{\textnormal{th}}$-order Taylor expansion of $F$ and a convex and bounded set $\XCal$ can be computed approximately in an efficient manner (see Section~\ref{sec:DE} for details). 
\end{itemize}
For the first-order VI methods (i.e., $p=1$),~\citet{Nemirovski-2004-Prox} has proved that the extragradient (EG) method~\citep{Korpelevich-1976-Extragradient,Antipin-1978-Method} converges to a weak solution with a global rate of $O(\epsilon^{-1})$ if $F$ is monotone and Eq.~\eqref{eq:smooth} holds. There are other methods with the same global rate guarantee, including forward-backward splitting method~\citep{Tseng-2000-Modified}, optimistic gradient (OG) method~\citep{Popov-1980-Modification,Mokhtari-2020-Convergence,Kotsalis-2022-Simple} and dual extrapolation method~\citep{Nesterov-2007-Dual}. All these methods match the lower bound of~\citet{Ouyang-2021-Lower} and are thus optimal. In addition, a general adaptive line search framework has been proposed to unify and extend several convergence results from the VI literature~\citep{Magnanti-2004-Solving}. 

The investigation of second-order and high-order ($p \geq 2$) counterparts of these first-order methods is less advanced, as exploiting high-order derivative information is much more involved for VIs~\citep{Nesterov-2006-Constrained,Monteiro-2012-Iteration}. Aiming to fill this gap, some work has been recently devoted to studying high-order extensions of first-order VI methods~\citep{Bullins-2022-Higher,Lin-2023-Monotone,Jiang-2022-Generalized}. These extensions attain a rate of $O(\epsilon^{-2/(p+1)}\log\log(1/\epsilon))$ but require a nontrivial line-search procedure at each iteration. Although the $\log\log(1/\epsilon)$ factor is modest by itself, it reflects complexity in current design of high-order VI methods, a complexity which might hinder practical application. Notably,~\citet[page 305]{Nesterov-2018-Lectures} emphasized the difficulty of removing the line search procedure without sacrificing the global rate of convergence and highlighted the goal of obtaining a simple and optimal high-order method as an open and challenging question. We summarize the challenge as follows:
\begin{center}
\textbf{Can we design a simple and optimal $p^{\textnormal{th}}$-order VI method without line search?}
\end{center}
In this paper, we present an affirmative answer to this query by identifying a $p^{\textnormal{th}}$-order method that achieves a global rate of $O(\epsilon^{-2/(p+1)})$ while dispensing entirely with the line-search inner loop. The core idea of the proposed method is to incorporate a simple adaptive strategy into a high-order generalization of the dual extrapolation method. 

There are two main reasons why we choose the dual extrapolation method as a base algorithm for our high-order methods. First, the dual extrapolation method has its own merits as summarized in~\citet{Nesterov-2007-Dual}, and the second-order VI method to attain a global convergence rate of $O(\epsilon^{-1})$~\citep{Nesterov-2006-Constrained} was firstly developed based on a dual extrapolation step. Our method can be interpreted as an adaptive variant of this method (see Section~\ref{subsec:NDE}). Second, the dual extrapolation step is an important ingredient for algorithm design in optimization, given the close relationship between extrapolation and acceleration in the context of several first-order methods for smooth convex optimization~\citep{Lan-2018-Optimal,Lan-2018-Random}. This is in contrast to the EG method, which is an approximate proximal point method~\citep{Mokhtari-2020-Convergence}. It would deepen our understanding of dual extrapolation if we could design a simple and optimal high-order VI method based on this scheme. 

\paragraph{Contributions.} The contribution of this paper consists in fully closing the gap between the upper and lower bounds in the monotone setting and improving the state-of-the-art upper bounds in the strongly monotone and/or structured non-monotone settings. In further detail: 
\begin{enumerate}
\item We present a new $p^{\textnormal{th}}$-order method for solving smooth and monotone VIs where $F$ has a Lipschitz continuous $(p-1)^{\textnormal{th}}$-order derivative and $\XCal$ is convex and bounded.  We prove that the number of calls of subproblem solvers required by our method to find an $\epsilon$-weak solution is bounded by 
\begin{equation*}
O\left(\left(\frac{LD^{p+1}}{\epsilon}\right)^{\frac{2}{p+1}}\right). 
\end{equation*}
We prove that our $p^{\textnormal{th}}$-order method is indeed optimal by establishing a matching lower bound of $\Omega(\epsilon^{-2/(p+1)})$ under a generalized linear span assumption. We propose a restarted version of our $p^{\textnormal{th}}$-order method for solving smooth and $p^{\textnormal{th}}$-order uniformly monotone VIs~\citep{Bauschke-2017-Convex}. In particular, these are problems in which there exists a constant $\mu > 0$ such that\footnote{We refer to~\citet[Chapter~22]{Bauschke-2017-Convex} for a more general definition of uniformly monotone operators and relevant discussions. This class of operators are closely related to a direct generalization of uniformly convex functions~\citep[Section 2]{Nesterov-2008-Accelerating}.} 
\begin{equation*}
\langle F(x) - F(x'), x- x'\rangle \geq \mu\|x - x'\|^{p+1}, \quad \textnormal{for all } x, x' \in \br^d.  
\end{equation*}
We show that the number of calls of subproblem solvers required to find $\hat{x} \in \XCal$ satisfying $\|\hat{x} - x^\star\| \leq \epsilon$ is bounded by 
\begin{equation*}
O\left(\kappa^{\frac{2}{p+1}}\log_2\left(\frac{D}{\epsilon}\right)\right),
\end{equation*}
where $\kappa = L/\mu$ refers to the condition number of $F$. Focusing on smooth and strongly monotone VIs, where there exists a constant $\mu > 0$ such that 
\begin{equation*}
\langle F(x) - F(x'), x- x'\rangle \geq \mu\|x - x'\|^2, \quad \textnormal{for all } x, x' \in \br^d, 
\end{equation*}
we show that another restarted version of our $p^{\textnormal{th}}$-order method can achieve a local superlinear rate for the case of $p \geq 2$. 

\item We show how to modify our framework such that it can be used for solving smooth and nonmonotone VIs satisfying the so-called Minty condition (see Definition~\ref{def:Minty}).  Again, we note that a line-search procedure is not required. We prove that the number of calls of subproblem solvers to find an $\epsilon$-strong solution is bounded by 
\begin{equation*}
O\left(\left(\frac{LD^{p+1}}{\epsilon}\right)^{\frac{2}{p}}\right). 
\end{equation*}
Two restarted version of our $p^{\textnormal{th}}$-order method attain a global linear rate under additional $p^{\textnormal{th}}$-order uniform Minty condition and a local superlinear rate (for the case of $p \geq 2$) under additional strong Minty condition. 
\end{enumerate}
\paragraph{Comparison to~\citet{Adil-2022-Optimal}.} Concurrently appearing on arXiv,~\citet{Adil-2022-Optimal} has established the same upper bounds as ours for a high-order generalization of the EG method for solving smooth and monotone VIs. Their method was later extended by two subsequent works to solve strongly monotone VIs~\citep{Huang-2022-Approximation} and nonmonotone VIs satisfying the Minty condition~\citep{Huang-2023-Beyond}. It is also worth remarking that the methods from~\citet{Huang-2022-Approximation,Huang-2023-Beyond} are based on similar restarting strategies but their refined analysis leads to a better convergence rate guarantee for solving strongly monotone VIs (up to a log factor) (see the discussion after~\citet[Theorem 3.2]{Huang-2022-Approximation}). 

A lower bound has been established in~\citet{Adil-2022-Optimal} for \textit{a class of $p^{\textnormal{th}}$-order methods restricted to solving the primal problem}. This is a rather strong limitation that excludes both our method and their method. We derive the same lower bound for a broader class of $p^{\textnormal{th}}$-order methods that include both our method and their method thanks to the construction of a new hard instance. Although the hard instance function is different (and the lower bound does improve), we do wish to acknowledge that the proof techniques from~\citet{Adil-2022-Optimal} inspired our analysis. 

\paragraph{Comparison to~\citet{Bullins-2022-Higher,Lin-2023-Monotone,Jiang-2022-Generalized}.} Prior to our work and the concurrent work of~\citet{Adil-2022-Optimal}, all existing high-order VI methods have been designed based on line-search procedures and are shown to achieve a rate of $O(\epsilon^{-2/(p+1)}\log\log(1/\epsilon))$ for solving smooth and monotone VIs. In this context,~\citet{Bullins-2022-Higher} was the first to prove improved rates for solving monotone VIs with third-order smoothness and beyond. Their method requires an oracle for finding a fixed point of a nonlinear equation using an implicit update. This necessities a nontrivial line-search procedure per iteration and leads to a global rate of $O(\epsilon^{-2/(p+1)}\log(1/\epsilon))$. Subsequently,~\citet{Lin-2023-Monotone} investigated the dynamics of high-order VI methods from a continuous-time viewpoint by proposing a novel closed-loop control system. Their analysis offers a simplification of existing analyses in~\citet{Bullins-2022-Higher} as well as new results concerning high-order VI methods. However, the method from~\citet{Lin-2023-Monotone} still requires line search and the obtained rate is the same. By distilling the idea of optimism,~\citet{Jiang-2022-Generalized} proposed a generalized optimistic method with a novel adaptive line search procedure that provably only requires an $O(\log\log(1/\epsilon))$ calls to a subproblem solver per iteration on average. This leads to an improve rate of $O(\epsilon^{-2/(p+1)}\log\log(1/\epsilon))$ in monotone setting and $O((\kappa D^{p-1})^{\frac{2}{p+1}} + \log\log(1/\epsilon))$ in strongly monotone setting. While these line-search-based VI methods are indeed an achievement and can be amenable to implementation, it would be important to understand whether or not there exists a simple and optimal high-order VI method that has no need for a line-search procedure. 

In comparison to~\citet{Bullins-2022-Higher,Lin-2023-Monotone,Jiang-2022-Generalized}, we believe that our method offers advantages in terms of simplicity. From an algorithmic design viewpoint, we incorporate a simple adaptive strategy into a high-order generalization of the dual extrapolation method, dispensing entirely with the line-search inner loop. Regarding technical parts, our convergence analysis only depends on a simple Lyapunov function and is easy to understand. However, it is worth remarking that our work does not eliminate the potential advantages of using line search. In fact, while the optimal first-order VI methods do not require line search,~\citet{Magnanti-2004-Solving} showed the benefits of adaptive line search by proposing a general framework to unify and extend several convergence results from the literature. The preliminary numerical results also confirmed that the adaptive line search procedure is fast in practice yet not very stable~\citep{Lin-2022-Explicit}. Nonetheless, we believe it is promising to study the line search procedure from~\citet{Jiang-2022-Generalized} and see if modifications can speed up high-order VI methods in a universal manner. 

\paragraph{Further related work.} In addition to the aforementioned works, we review relevant research on high-order convex optimization. We focus on $p^{\textnormal{th}}$-order methods for $p \geq 2$.

The systematic investigation of the global convergence rate of second-order methods originates in work on  the cubic regularization of Newton's method (CRN)~\citep{Nesterov-2006-Cubic} and its accelerated counterpart (ACRN)~\citep{Nesterov-2008-Accelerating}. The ACRN method was then extended with a $p^{\textnormal{th}}$-order regularization model, yielding an improved global convergence rate of $O(\epsilon^{-1/(p+1)})$~\citep{Baes-2009-Estimate}, while an adaptive $p^{\textnormal{th}}$-order method was proposed in~\citet{Jiang-2020-Unified} with the same global rate guarantee. This extension was recently revisited by other works~\citep{Nesterov-2021-Implementable,Grapiglia-2023-Adaptive} with a discussion on an efficient implementation of a third-order method. Meanwhile, within the accelerated Newton proximal extragradient (ANPE) framework~\citep{Monteiro-2013-Accelerated}, a $p^{\textnormal{th}}$-order method was also proposed by~\citet{Gasnikov-2019-Near} with a global convergence rate of $O(\epsilon^{-2/(3p+1)}\log(1/\epsilon))$ for minimizing a convex function whose the $p^{\textnormal{th}}$-order derivative is Lipschitz continuous. An additional log factor remains between the best known upper bound and the lower bound of $O(\epsilon^{-2/(3p+1)})$~\citep{Arjevani-2019-Oracle}. This gap has been closed by two independent works~\citep{Kovalev-2022-First,Carmon-2022-Optimal} that offer a complementary viewpoint to~\citet{Monteiro-2013-Accelerated,Gasnikov-2019-Near} on how to remove the line-search scheme. Subsequently, the $p^{\textnormal{th}}$-order ANPE framework was extended to a strongly convex setting~\citep{Marques-2022-Variants} and was shown to achieve a global linear rate and a local superlinear rate while the lower bound on deterministic $p^{\textnormal{th}}$-order methods for minimizing a smooth and strongly convex function was established in~\citet{Kornowski-2020-High}. Beyond the setting with Lipschitz continuous $p^{\textnormal{th}}$-order derivatives, these $p^{\textnormal{th}}$-order methods have been adapted to a setting with H\"{o}lder continuous $p^{\textnormal{th}}$-order derivatives~\citep{Grapiglia-2017-Regularized,Grapiglia-2019-Accelerated,Grapiglia-2020-Tensor,Song-2021-Unified,Doikov-2022-Local}. Further settings include smooth nonconvex minimization~\citep{Cartis-2010-Complexity,Cartis-2011-Adaptive-I,Cartis-2011-Adaptive-II,Cartis-2019-Universal,Birgin-2016-Evaluation,Birgin-2017-Worst,Martinez-2017-High} as well as structured nonsmooth minimization~\citep{Bullins-2020-Highly}. There is also a complementary line of research that studies the favorable properties of lower-order methods in the setting of higher-order smoothness~\citep{Nesterov-2021-Auxiliary,Nesterov-2021-Inexact, Nesterov-2021-Superfast}. 

We are aware of various high-order methods obtained via discretization of continuous-time dynamical systems~\citep{Wibisono-2016-Variational,Lin-2022-Control}. In particular,~\citet{Wibisono-2016-Variational} showed that the ACRN method and its $p^{\textnormal{th}}$-order variants can be obtained from implicit discretization of an open-loop system without Hessian-driven damping.~\citet{Lin-2022-Control} have provided a control-theoretic perspective on $p^{\textnormal{th}}$-order ANPE methods by recovering them from implicit discretization of a closed-loop system with Hessian-driven damping. Both of these two works proved the convergence rate of $p^{\textnormal{th}}$-order ACRN and ANPE methods using Lyapunov functions.  

\paragraph{Organization.} In Section~\ref{sec:prelim}, we present the setup for variational inequality (VI) problems and provide definitions for the class of operators and optimality criteria we consider in this paper. In addition, we review the dual extrapolation method. In Section~\ref{sec:DE}, we present our new method, its restarted version, and our main results on the global and local convergence guarantee for monotone and nonmonotone VIs. We also establish a matching lower bound for a broad class of $p^{\textnormal{th}}$-order methods in the monotone setting. In Section~\ref{sec:proof}, we provide the proofs for our results. In Section~\ref{sec:conclu}, we conclude the paper with a discussion on future research directions.

\paragraph{Notation.} We use lower-case letters such as $x$ to denote vectors and upper-case letters such as $X$ to denote tensors.  Let $\br^d$ be a finite-dimensional Euclidean space (the dimension is $d \in \{1, 2, \ldots\}$), endowed with the scalar product $\langle \cdot, \cdot\rangle$. For $x \in \br^d$, we let $\|x\|$ denote its $\ell_2$-norm.  For $X \in \br^{d_1 \times \ldots \times d_p}$, we define 
\begin{equation*}
X[z^1, \cdots, z^p] = \sum_{1 \leq i_j \leq d_j, 1 \leq j \leq p} (X_{i_1, \cdots, i_p})z_{i_1}^1 \cdots z_{i_p}^p, 
\end{equation*}
and $\|X\|_\op = \max_{\|z^i\|=1, 1 \leq j \leq p} X[z^1, \cdots, z^p]$ as well. Fixing $p \geq 0$ and letting $F: \br^d \rightarrow \br^d$ be a continuous and high-order differentiable operator, we define $\nabla^{(p)} F(x)$ as the $p^{\textnormal{th}}$-order derivative at a point $x \in \br^d$ and write $\nabla^{(0)} F = F$. To be more precise, letting $z_1, \ldots, z_k \in \br^d$, we have
\begin{equation*}
\nabla^{(k)} F(x)[z^1, \cdots, z^k] = \sum_{1 \leq i_1, \ldots, i_k \leq d} \left(\frac{\partial F_{i_1}}{\partial x_{i_2} \cdots \partial x_{i_k}}(x)\right) z_{i_1}^1 \cdots z_{i_k}^k. 
\end{equation*}
For a closed and convex set $\XCal \subseteq \br^d$, we let $\proj_\XCal$ be the orthogonal projection onto $\XCal$ and let $\dist(x, \XCal) = \inf_{x' \in \XCal} \|x' - x\|$ denote the distance between $x$ and $\XCal$. Finally, $a = O(b(L, \mu, \epsilon))$ stands for an upper bound $a \leq C \cdot b(L, \mu, \epsilon)$, where $C > 0$ is independent of parameters $L, \mu$ and the tolerance $\epsilon \in (0, 1)$, and $a = \tilde{O}(b(L, \mu, \epsilon))$ indicates the same inequality where $C > 0$ depends on logarithmic factors of $1/\epsilon$.

\section{Preliminaries and Technical Background}\label{sec:prelim}
In this section, we present the basic formulation of variational inequality (VI) problems and provide definitions for the class of operators and optimality criteria considered in this paper. We further give a brief overview of Nesterov's dual extrapolation concept from which our new method originates.  

\subsection{Variational inequality problem}
The regularity conditions that we consider for $F: \br^d \rightarrow \br^d$ are as follows.
\begin{definition}\label{def:smooth}
$F$ is \emph{$k^{\textnormal{th}}$-order $L$-smooth} if 
\begin{equation*}
\|\nabla^{(k)} F(x) -  \nabla^{(k)} F(x')\|_\op \leq L\|x - x'\|,
\end{equation*}
for all $x, x'$.
\end{definition} 
\begin{definition}\label{def:monotone}
We have the following characterizations: 
\begin{itemize}
\item $F$ is monotone if $\langle F(x) - F(x'), x - x'\rangle \geq 0$ for all $x, x'$. 
\item $F$ is $k^{\textnormal{th}}$-order $\mu$-uniformly monotone if $\langle F(x) - F(x'), x - x'\rangle \geq \mu \|x - x'\|^{k+1}$ for all $x, x'$. 
\item $F$ is $\mu$-strongly monotone if $\langle F(x) - F(x'), x - x'\rangle \geq \mu \|x - x'\|^2$ for all $x, x'$. 
\end{itemize}
\end{definition}
With the definitions in mind, we state the assumptions that impose in addition to \textbf{A1} and \textbf{A2} in order to define highly smooth VI problems. 
\begin{assumption}\label{Assumption:smooth}
We assume that (i) $F: \br^d \rightarrow \br^d$ is $(p-1)^{\textnormal{th}}$-order $L$-smooth, and (ii) $\XCal$ is convex and bounded with a diameter $D = \max_{x, x' \in \XCal} \|x - x'\| > 0$. 
\end{assumption}
The convergence of derivative-based optimization methods to a weak solution $x^\star \in \XCal$ depends on properties of $F$ near this point, and in particular some form of smoothness condition is needed. As for the boundedness condition for $\XCal$, it is standard in the VI literature~\citep{Facchinei-2007-Finite}. This condition not only guarantees the validity of the most natural optimality criterion in the monotone setting---the gap function~\citep{Nemirovski-2004-Prox,Nesterov-2007-Dual}---but additionally it is satisfied in real application problems~\citep{Facchinei-2007-Finite}. On the other hand, there is another line of work focusing on relaxing the boundedness condition via appeal to other notions of approximate solution~\citep{Monteiro-2010-Complexity,Monteiro-2011-Complexity,Monteiro-2012-Iteration,Chen-2017-Accelerated}. For simplicity, we retain the boundedness condition and leave the analysis for cases with unbounded constraint sets to future work. 

\paragraph{Monotone setting.}  For some of our results we focus on operators $F$ that are monotone in addition to Assumption~\ref{Assumption:smooth}. Under monotonicity, it is well known that any $\epsilon$-strong solution is an $\epsilon$-weak solution but the reverse does not hold true in general. Accordingly, we formally define $\hat{x} \in \XCal$ as an $\epsilon$-weak solution or an $\epsilon$-strong solution as follows:
\begin{equation*}
\begin{array}{lll}
\textbf{($\epsilon$-weak solution)} & \langle F(x), \hat{x} - x\rangle \leq \epsilon, & \textnormal{for all } x \in \XCal, \\
\textbf{($\epsilon$-strong solution)} & \langle F(\hat{x}), \hat{x} - x\rangle \leq \epsilon, & \textnormal{for all } x \in \XCal. 
\end{array}
\end{equation*}
These definitions motivate the use of a gap function, $\textsc{gap}(\cdot): \XCal \rightarrow \br_+$, defined by 
\begin{equation}\label{eq:cc-gap}
\textsc{gap}(\hat{x}) = \sup_{x \in \XCal} \ \langle F(x),  \hat{x} - x\rangle, 
\end{equation}
to measure the optimality of a point $\hat{x} \in \XCal$ that is output by various iterative solution methods; see~\citep[e.g.,][]{Tseng-2000-Modified,Nemirovski-2004-Prox,Nesterov-2007-Dual,Mokhtari-2020-Convergence}. The boundedness of $\XCal$ and the existence of a strong solution guarantee that the gap function is well defined. Formally, we have
\begin{definition}\label{def:cc-gap}
A point $\hat{\x} \in \XCal$ is an \emph{$\epsilon$-weak solution} to the monotone VI that corresponds to $F: \br^d \rightarrow \br^d$ and $\XCal \subseteq \br^d$ if we have $\textsc{gap}(\hat{x}) \leq \epsilon$. If $\epsilon = 0$, then $\hat{x} \in \XCal$ is a \emph{weak solution}.  
\end{definition}
In the strongly monotone setting, we let $\mu > 0$ denote the modulus of strong monotonicity for $F$. Under Assumption~\ref{Assumption:smooth}, we define $\kappa := L/\mu$ as the \textit{condition number} of $F$. It is worth mentioning that the condition number quantifies the difficulty of solving the optimization problem~\citep{Nesterov-2018-Lectures} and appears in the iteration complexity bound of derivative-based methods for optimizing a smooth and strongly convex function. Accordingly, the VI that corresponds to $F$ and $\XCal$ is more computationally challenging as $\kappa > 0$ increases. 

\paragraph{Structured nonmonotone setting.} We study the case where $F$ is nonmonotone but satisfies the Minty condition.  Imposing such a condition is crucial since the smoothness of $F$ is not sufficient to guarantee that the problem is computationally tractable. This was shown by~\citet{Daskalakis-2021-Complexity} who established that deciding whether an approximate min-max solution exists is NP hard in smooth and nonconvex-nonconcave min-max optimization (which is a special instance of nonmonotone VIs). 

A line of recent works has shown that the nonmonotone VI problem satisfying the Minty condition is computationally tractable~\citep{Solodov-1999-New,Dang-2015-Convergence,Iusem-2017-Extragradient,Kannan-2019-Optimal,Song-2020-Optimistic,Liu-2021-First,Diakonikolas-2021-Efficient}. We thus make the following formal definition.
\begin{definition}\label{def:Minty}
The VI corresponding to $F: \br^d \rightarrow \br^d$ and $\XCal \subseteq \br^d$ satisfies the \emph{Minty condition} if there exists a point $x^\star \in \XCal$ such that $\langle F(x), x- x^\star\rangle \geq 0$ for all $x \in \XCal$. 
\end{definition}
We make some comments on the Minty condition. First, this condition simply assumes the existence of at least one weak solution. Second,~\citet[Theorem~3.1]{Harker-1990-Finite} guarantees that there is at least one strong solution since $F$ is continuous and $\XCal$ is closed and bounded. However, the set of weak solutions is only a subset of the set of strong solutions if $F$ is not necessarily monotone, and the weak solution might not exist. From this perspective, the Minty condition gives a favorable structure. Furthermore, the Minty condition is weaker than generalized monotone assumptions~\citep{Dang-2015-Convergence,Iusem-2017-Extragradient,Kannan-2019-Optimal} that imply that the computation of an $\epsilon$-strong solution of nonmonotone VIs is tractable for first-order methods. Finally, we say the VI satisfies the \emph{$p^{\textnormal{th}}$-order $\mu$-uniform Minty condition} if there exists a point $x^\star \in \XCal$ such that $\langle F(x), x- x^\star\rangle \geq \mu\|x - x^\star\|^{p+1}$ for all $x \in \XCal$, and the VI satisfies the \emph{$\mu$-strong Minty condition}~\citep{Song-2020-Optimistic} if there exists a point $x^\star \in \XCal$ such that $\langle F(x), x- x^\star\rangle \geq \mu\|x - x^\star\|^2$ for all $x \in \XCal$. 

Accordingly, we define the \emph{residue function} $\textsc{res}(\cdot): \XCal \rightarrow \br_+$ given by 
\begin{equation}\label{eq:cc-residue}
\textsc{res}(\hat{x}) = \sup_{x \in \XCal} \ \langle F(\hat{x}),  \hat{x} - x\rangle, 
\end{equation}
which measures the optimality of a point $\hat{x} \in \XCal$ achieved by iterative solution methods; see~\citep[e.g.,][]{Dang-2015-Convergence,Iusem-2017-Extragradient,Kannan-2019-Optimal,Song-2020-Optimistic}. It is worth noting that the boundedness of $\XCal$ and the continuity of $F$ guarantee that the residue function is well defined. Formally, we have
\begin{definition}\label{def:cc-res}
A point $\hat{\x} \in \XCal$ is an \emph{$\epsilon$-strong solution} to the nonmonotone VI corresponding to $F: \br^d \rightarrow \br^d$ and $\XCal \subseteq \br^d$ if we have $\textsc{res}(\hat{x}) \leq \epsilon$. If $\epsilon = 0$, then $\hat{x} \in \XCal$ is a \emph{strong solution}.  
\end{definition}
There are many application problems that can be formulated as nonmonotone VIs satisfying the Minty condition, such as competitive exchange economies~\citep{Brighi-2002-Characterizations} and product pricing~\citep{Choi-1990-Product,Gallego-2014-Dynamic,Ewerhart-2014-Cournot}. In addition, the Minty condition restricted to nonconvex optimization was adopted for analyzing the convergence of stochastic gradient descent for deep learning~\citep{Li-2017-Convergence} and it has found real-world applications~\citep{Kleinberg-2018-Alternative}. 
\paragraph{Comments on weak versus strong solutions.} First, the montonicity assumption is assumed such that the \textit{averaged iterates} make sense and we have proved that the averaged iterates converge to an $\epsilon$-weak solution with a faster convergence rate of $O(\epsilon^{-2/(p+1)})$ in this setting (see Theorem~\ref{Thm:monotone-global}). Such a bound is stronger than that for convergence rate of \textit{best iterates} to an $\epsilon$-strong solution under only the Minty condition (see Theorem~\ref{Thm:nonmonotone-global}). Further, if we impose the monotonicity assumption, we conjecture that the rate of convergence to an $\epsilon$-strong solution can be improved from $O(\epsilon^{-2/p})$ to $O(\epsilon^{-2/(p+1)})$. Such a result has been achieved for the case of $p=1$~\citep{Diakonikolas-2020-Halpern}. However, it is worth mentioning that the first-order method in~\citet{Diakonikolas-2020-Halpern} is different from the first-order extragradient method and the first-order dual extrapolation method which are known to achieve an optimal convergence to an $\epsilon$-weak solution. It remains unclear how to design a high-order generalization of such new Halpern iteration methods. Finally, the complexity bound of $O(\epsilon^{-2/(p+1)})$ can not be extended beyond the monotone setting if only the Minty condition holds. Indeed, the key ingredient for proving the complexity bound of $O(\epsilon^{-2/(p+1)})$ is the use of averaged iterates in our new method. Such an averaging technique is known to be crucial for the monotone setting~\citep{Magnanti-1997-Averaging} but is not known to be valid when only the Minty condition holds. In addition, the fast convergence of Halpern iteration in~\citet{Diakonikolas-2020-Halpern} for achieving an $\epsilon$-strong solution heavily relies on the monotonicity assumption and does not extend to the setting when only the Minty condition holds. We would be very surprised if the optimal complexity bound for the monotone setting (note that we have established the matching lower bound) can be achieved for the setting when only the Minty condition holds. Even for the case of $p=1$, we are not aware of any relevant supporting evidence. Further exploration of this topic is beyond the scope of our paper. 

\paragraph{Comments on Euclidean versus non-Euclidean settings.} The non-Euclidean generalization of the first-order dual extrapolation method has been shown to outperform the original method in various application problem (e.g., the case where $\XCal$ is a simplex)~\citep{Nesterov-2007-Dual}. It remains a possibility that such a benefit also occurs for the case of $p \geq 2$ and thus it seems promising to study the high-order dual extrapolation method in non-Euclidean settings. In fact, we can follow the approach from~\citet{Adil-2022-Optimal} and extend our methods to the non-Euclidean setting using Bregman divergence. However, we can not say much about the superiority of high-order dual extrapolation methods in the non-Euclidean setting since the solution of the subproblem will become much more involved. This is different from the first-order case where each subproblem has a closed-form solution even in non-Euclidean settings. This is also a intriguing topic but again beyond the scope of our paper. 

\subsection{Nesterov's dual extrapolation method}\label{subsec:NDE}
Nesterov’s dual extrapolation method~\citep{Nesterov-2007-Dual} was shown to be an optimal first-order method for computing the weak solution of the VI when $F$ is zeroth-order $L$-smooth and monotone~\citep{Ouyang-2021-Lower}. We recall the basic formulation in our setting of a VI defined via an operator $F: \br^d \rightarrow \br^d$ and a closed, convex and bounded set $\XCal \subseteq \br^d$. Starting with the initial points $x_0 \in \XCal$ and $s_0 = 0 \in \br^d$, the $k^{\textnormal{th}}$ iteration of the scheme is given by ($k \geq 1$): 
\begin{equation*}
\begin{array}{rl}
& \textnormal{Find } v_k \in \XCal \textnormal{ s.t. } v_k = \argmax_{v \in \XCal} \langle s_{k-1}, v - x_0\rangle - \tfrac{\beta}{2}\|v - x_0\|^2, \\
& \textnormal{Find } x_k \in \XCal \textnormal{ s.t. } \langle F(v_k) + \beta(x_k - v_k), x - x_k\rangle \geq 0 \textnormal{ for all } x \in \XCal, \\
& s_k = s_{k-1} - \lambda F(x_k). 
\end{array}
\end{equation*}
This method can be viewed as an instance of the celebrated extragradient method in the dual space (we refer to $s \in \br^d$ as the dual variable). Indeed, the rule which transforms a point $s_{k-1}$ into the next point $s_k$ at the $k^{\textnormal{th}}$ iteration is called a dual extrapolation step.~\citet[Theorem~2]{Nesterov-2007-Dual} showed that the dual extrapolation method, with $\beta = L$ and $\lambda = 1$, generates a sequence $\{x_k\}_{k \geq 0} \subseteq \XCal$ satisfying the condition that the average iterate, $\tilde{x}_k = \frac{1}{k+1}\sum_{i=0}^k x_i$, is an $\epsilon$-weak solution after at most $O(\epsilon^{-1})$ iterations. Here, $L > 0$ is the Lipschitz constant of $F$. 

Nesterov also considered the setting where $F$ is monotone and first-order $L$-smooth and proposed a second-order dual extrapolation method for computing the weak solution of the VI~\citep{Nesterov-2006-Constrained}. Starting with the initial points $x_0 \in \XCal$ and $s_0 = 0 \in \br^d$, the $k^{\textnormal{th}}$ iteration of the scheme is given by ($k \geq 1$):
\begin{equation*}
\begin{array}{rl}
& \textnormal{Find } v_k \in \XCal \textnormal{ s.t. } v_k = \argmax_{v \in \XCal} \langle s_{k-1}, v - x_0\rangle - \tfrac{\beta}{3}\|v - x_0\|^3, \\
& \textnormal{Find } x_k \in \XCal \textnormal{ s.t. } \langle F_{v_k}^1(x_k) + \tfrac{M}{2}\|x_k - v_k\|(x_k - v_k), x - x_k\rangle \geq 0 \textnormal{ for all } x \in \XCal, \\
& s_k = s_{k-1} - \lambda F(x_k),
\end{array}
\end{equation*}
where $F^1_v(\cdot): \br^d \rightarrow \br^d$ is defined as a first-order Taylor expansion of $F$ at a point $v \in \XCal$:
\begin{equation*}
F^1_v(x) = F(v) + \nabla F(v)(x - v). 
\end{equation*}
This scheme is based on the dual extrapolation step which combines a different regularization with a first-order Taylor expansion of $F$. This makes sense since we have zeroth-order and first-order derivative information available and hope to use both of them to accelerate convergence.  Similar ideas have been studied for convex optimization~\citep{Nesterov-2006-Cubic}, leading to a simple second-order method with a faster global rate~\citep{Nesterov-2008-Accelerating} than optimal first-order methods~\citep{Nesterov-1983-Method}. Unfortunately, the second-order dual extrapolation method with $\beta = 6L$, $M = 5L$ and $\lambda = 1$ is only guaranteed to achieve an iteration complexity of $O(\epsilon^{-1})$~\citep[Theorem~4]{Nesterov-2006-Constrained}.

\section{A Regularized High-Order Model and Algorithm}\label{sec:DE}
In this section, we present our algorithmic derivation of \textsf{Perseus} and provide a theoretical convergence guarantee for the method. We provide intuition into why \textsf{Perseus} and its restarted version yield the fast rates of convergence for VI problems. We present a complete treatment of the global and local convergence of \textsf{Perseus} and several restarted versions of our method for both the monotone setting and the nonmonotone setting under the Minty condition. 

\subsection{Algorithmic scheme}
We present our $p^{\textnormal{th}}$-order method---\textsf{Perseus}($p$, $x_0$, $L$, $T$, \textsf{opt})---in Algorithm~\ref{Algorithm:DE}.  Here $p \in \{1, 2, \ldots\}$ is the order, $x_0 \in \XCal$ is an initial point, $L > 0$ is a Lipschitz constant for $(p-1)^{\textnormal{th}}$-order smoothness, $T$ is the maximum iteration number and $\textsf{opt} \in \{0, 1, 2\}$ is the type of output. Our method is a generalization of the dual extrapolation method~\citep{Nesterov-2007-Dual} from first order to general $p^{\textnormal{th}}$ order.

The novelty of our method lies in an adaptive strategy used for updating $\lambda_{k+1}$ (see \textbf{Step 4}). This modification is simple yet important. It is the key for obtaining a global rate of $O(\epsilon^{-2/(p+1)})$ (monotone) and that of $O(\epsilon^{-2/p})$ (nonmonotone with the Minty condition) under Assumption~\ref{Assumption:smooth}. Focusing on the case of $p=2$ and the monotone setting, our results improve on the best existing global convergence rates of $O(\epsilon^{-1})$~\citep{Nesterov-2006-Constrained} and that of $O(\epsilon^{-2/3}\log\log(1/\epsilon))$~\citep{Monteiro-2012-Iteration} under Assumption~\ref{Assumption:smooth}, while not sacrificing algorithmic simplicity. In addition, our methods allow the subproblem to be solved inexactly, and we give options for choosing the type of outputs under different assumptions. 
\begin{algorithm}[!t]
\begin{algorithmic}\caption{\textsf{Perseus}($p$, $x_0$, $L$, $T$, $\textsf{opt}$)}\label{Algorithm:DE}
\STATE \textbf{Input:} order $p$, initial point $x_0 \in \XCal$, parameter $L$, iteration number $T$ and $\textsf{opt} \in \{0, 1, 2\}$. 
\STATE \textbf{Initialization:} set $s_0 = 0_d \in \br^d$.
\FOR{$k = 0, 1, 2, \ldots, T$} 
\STATE \textbf{STEP 1:} If $x_k \in \XCal$ is a solution of the VI, then \textbf{stop}. 
\STATE \textbf{STEP 2:} Compute $v_{k+1} = \argmax_{v \in \XCal} \{\langle s_k, v - x_0\rangle - \frac{1}{2}\|v - x_0\|^2\}$. 
\STATE \textbf{STEP 3:} Compute $x_{k+1} \in \XCal$ such that Eq.~\eqref{condition:approximation} holds true. 
\STATE \textbf{STEP 4:} Compute $\lambda_{k+1} > 0$ such that $\frac{1}{20p-8} \leq \tfrac{\lambda_{k+1} L \|x_{k+1} - v_{k+1}\|^{p-1}}{p!} \leq \frac{1}{10p+2}$. 
\STATE \textbf{STEP 5:} Compute $s_{k+1} = s_k - \lambda_{k+1} F(x_{k+1})$. 
\ENDFOR
\STATE \textbf{Output:} $\hat{x} = \left\{\begin{array}{cl}
\tilde{x}_T = \frac{1}{\sum_{k=1}^T \lambda_k}\sum_{k=1}^T \lambda_k x_k, & \textnormal{if } \textsf{opt} = 0, \\
x_T, & \textnormal{else if } \textsf{opt} = 1, \\ 
x_{k_T} \textnormal{ for } k_T = \argmin_{1 \leq k \leq T} \|x_k - v_k\|, & \textnormal{else if } \textsf{opt} = 2.
\end{array}\right. $
\end{algorithmic}
\end{algorithm}
\paragraph{Comments on inexact subproblem solving.} We remark that \textbf{Step 3} involves computing an approximate strong solution to the VI where we define the operator $F_{v_{k+1}}(x)$ as the sum of a high-order polynominal and a regularization term. Indeed, we have\footnote{For ease of presentation, we choose the factor of 5 here. It is worth noting that other large coefficients also suffice to achieve the same global convergence rate guarantee. }
\begin{eqnarray*}
\lefteqn{F_{v_{k+1}}(x) = F(v_{k+1}) + \langle \nabla F(v_{k+1}), x-v_{k+1}\rangle} \\ 
& & + \ldots + \tfrac{1}{(p-1)!}\nabla^{(p-1)} F(v_{k+1})[x-v_{k+1}]^{p-1} + \tfrac{5L}{(p-1)!}\|x - v_{k+1}\|^{p-1}(x - v_{k+1}), 
\end{eqnarray*}
where we write the VI of interest in the subproblem as follows:
\begin{equation}\label{def:subproblem}
\textnormal{Find } x_{k+1} \in \XCal \textnormal{ such that } \langle F_{v_{k+1}}(x_{k+1}), x - x_{k+1} \rangle \geq 0 \textnormal{ for all } x \in \XCal. 
\end{equation}
Since $F_{v_{k+1}}$ is continuous and $\XCal$ is convex and bounded,~\citet[Theorem~3.1]{Harker-1990-Finite} guarantees that a strong solution to the VI in Eq.~\eqref{def:subproblem} exists and the problem of finding an approximate strong solution is well defined. 

In the monotone setting, we can prove that the $p^\textnormal{th}$-order regularization subproblem in Eq.~\eqref{def:subproblem} is monotone (in fact, it is relatively strongly monotone) if the original VI is $p^\textnormal{th}$-order $L$-smooth and monotone. Indeed, the VI with $F$ is monotone if and only if the symmetric part of the Jacobian matrix $\nabla F(x)$ is \textit{positive semidefinite} for all $x \in \br^d$~\citep[Proposition~12.3]{Rockafellar-2009-Variational}. That is to say, 
\begin{equation*}
\tfrac{1}{2}(\nabla F(x) + \nabla F(x)^\top) \succeq 0_{d \times d}, \quad \textnormal{for all } x \in \br^d. 
\end{equation*}
For the case of $p=1$, we have $\nabla F_{v_{k+1}}(x) = 5L \cdot I_{d \times d} \succeq 0_{d \times d}$ for all $x \in \br^d$ where $I_{d \times d} \in \br^{d \times d}$ is an identity matrix. Thus, the VI in Eq.~\eqref{def:subproblem} is $5L$-strongly monotone. For the case of $p \geq 2$, we have
\begin{eqnarray*}
\lefteqn{\nabla F_{v_{k+1}}(x) = \nabla F(v_{k+1}) + \ldots + \tfrac{1}{(p-2)!}\nabla^{(p-1)} F(v_{k+1})[x-v_{k+1}]^{p-2}} \\
& & + \tfrac{5L}{(p-1)!}\|x - v_{k+1}\|^{p-1}I_{d \times d} + \tfrac{5L}{(p-2)!}\|x - v_{k+1}\|^{p-2}(x - v_{k+1})(x - v_{k+1})^\top. \end{eqnarray*}
Since the original VI is $p^\textnormal{th}$-order $L$-smooth, we obtain from~\citet[Eq.~(7)]{Jiang-2022-Generalized} that 
\begin{equation*}
\|\nabla F(x) - (\nabla F(v_{k+1}) + \ldots + \tfrac{1}{(p-2)!}\nabla^{(p-1)} F(v_{k+1})[x-v_{k+1}]^{p-2})\|_\op \leq \tfrac{L}{(p-1)!}\|x - v_{k+1}\|^{p-1}. 
\end{equation*}
This implies that 
\begin{eqnarray*}
\lefteqn{\tfrac{1}{2}(\nabla F_{v_{k+1}}(x) + \nabla F_{v_{k+1}}(x)^\top) \succeq \tfrac{1}{2}(\nabla F(x) + \nabla F(x)^\top)} \\
& & + \tfrac{4L}{(p-1)!}\|x - v_{k+1}\|^{p-1}I_{d \times d} + \tfrac{5L}{(p-2)!}\|x - v_{k+1}\|^{p-2}(x - v_{k+1})(x - v_{k+1})^\top \\
& \succeq & \tfrac{4L}{(p-1)!}(\|x - v_{k+1}\|^{p-1}I_{d \times d} + \|x - v_{k+1}\|^{p-2}(x - v_{k+1})(x - v_{k+1})^\top), 
\end{eqnarray*}
where the second inequality holds since the original VI is monotone. Thus, the VI in Eq.~\eqref{def:subproblem} is monotone and $4L$-relatively strongly monotone with respect to the reference function $h(x) = \tfrac{1}{p!}\|x - v_{k+1}\|^p$ (see~\citet{Nesterov-2021-Implementable} for the precise definition). Putting these pieces together yields the desired result. From a computational viewpoint, we can use the generalized mirror-prox method in~\citet{Titov-2022-Some} to compute $x_{k+1} \in \XCal$ satisfying the following approximation condition: 
\begin{equation}\label{condition:approximation}
\sup_{x \in \XCal} \ \langle F_{v_{k+1}}(x_{k+1}), x_{k+1} - x\rangle \leq \tfrac{L}{p!}\|x_{k+1} - v_{k+1}\|^{p+1}. 
\end{equation}
Thus, the solution of the subproblem in our framework is computationally tractable for the monotone setting.  Other efficient solvers have been developed for the case of $p = 2$ and $\XCal = \br^d$ in the context of optimization~\citep{Grapiglia-2021-Inexact} and minimax optimization~\citep{Huang-2022-Cubic,Adil-2022-Optimal,Lin-2022-Explicit} and were shown to be efficient in practice. However, it is worth mentioning that the subproblem solution method in our approach is different from that in existing line-search-based methods~\citep{Bullins-2022-Higher,Lin-2023-Monotone,Jiang-2022-Generalized}, and their per-iteration computational costs can not be directly compared. It remains an open challenge to develop a systematic benchmarking for these methods.

In the nonmonotone setting, the VI in Eq.~\eqref{def:subproblem} is not necessarily monotone and computing a solution $x_{k+1}$ satisfying Eq.~\eqref{condition:approximation} is intractable in general~\citep{Daskalakis-2021-Complexity}. However, $F_{v_{k+1}}$ is defined as the sum of a polynomial and a regularization term, and this special structure might lend itself to efficient numerical methods. For example, we consider the optimization setting where $F = \nabla f$ for a nonconvex function $f: \br^d \rightarrow \br$ with a Lipschitz second-order derivative, $\XCal = \br^d$ and $p = 2$. Solving the VI in Eq.~\eqref{def:subproblem} is equivalent to solving cubic regularization subproblems in unconstrained optimization: finding a global solution of the regularized polynomial in the following form of 
\begin{equation*}
\langle \nabla f(v_{k+1}), x - v_{k+1}\rangle + \tfrac{1}{2}\langle x - v_{k+1},  \nabla^2 f(v_{k+1})(x - v_{k+1})\rangle + \tfrac{L}{3}\|x - v_{k+1}\|^3. 
\end{equation*}
The above optimization problem is nonconvex but can be solved approximately in a provably efficient manner. Examples of cubic regularization solvers include some generalized conjugate gradient methods with the Lanczos process~\citep{Gould-1999-Solving,Gould-2010-Solving} and a simple variant of gradient descent~\citep{Carmon-2019-Gradient}. A recent textbook of~\citet{Cartis-2022-Evaluation} provides a detailed discussion of the existing techniques. The generalization of these techniques to handle the VI in Eq.~\eqref{def:subproblem} is challenging, however, and beyond the scope of this paper. 
\begin{algorithm}[!t]
\begin{algorithmic}\caption{\textsf{Perseus-restart}($p$, $x_0$, $L$, $\sigma$, $D$, $T$,  $\textsf{opt}$)}\label{Algorithm:restart}
\STATE \textbf{Input:} order $p$, initial point $x_0 \in \XCal$, parameters $L, \sigma, D$, iteration number $T$ and $\textsf{opt} \in \{0, 1\}$. 
\STATE \textbf{Initialization:} set $T_{\textnormal{inner}} = \left\{
\begin{array}{cl}
\lceil(\tfrac{2^{p+1}(5p-2)}{p!}\tfrac{LD^{p-1}}{\sigma})^{\frac{2}{p+1}}\rceil, & \textnormal{if } \textsf{opt} = 0, \\
1, & \textnormal{else if } \textsf{opt} = 1.
\end{array}\right.$
\FOR{$k = 0, 1, 2, \ldots, T$}
\STATE \textbf{STEP 1:} If $x_k \in \XCal$ is a solution of the VI, then \textbf{stop}. 
\STATE \textbf{STEP 2:} Compute $x_{k+1} = \textsf{Perseus}(p, x_k, L, T_{\textnormal{inner}}, \textsf{opt})$. 
\ENDFOR
\STATE \textbf{Output:} $x_{T+1}$. 
\end{algorithmic}
\end{algorithm}
\paragraph{Comments on adaptive strategies.} Our adaptive strategy for updating $\lambda_{k+1}$ was inspired by an in-depth consideration of the reason a nontrivial binary search procedure is needed in existing $p^{\textnormal{th}}$-order methods. These methods compute a pair, $\lambda_{k+1} > 0, x_{k+1} \in \XCal$, that (approximately) solve the $x$-subproblem that contains $\lambda$ and the $\lambda$-subproblem that contains $x$. In particular, the conditions can be written as follows:
\begin{equation*}
\begin{array}{rl}
& \alpha_{-} \leq \tfrac{\lambda_{k+1} L \|x_{k+1} - v_{k+1}\|^{p-1}}{p!} \leq \alpha_+ \textnormal{ for proper choices of } \alpha_- \textnormal{ and } \alpha_+, \\ 
& \langle F_{v_{k+1}}(x_{k+1}) + \tfrac{1}{\lambda_{k+1}}(x_{k+1} - v_{k+1}), x - x_{k+1} \rangle \geq 0 \textnormal{ for all } x \in \XCal, 
\end{array}
\end{equation*}
where 
\begin{equation}\label{def:approximation}
F_v(x) = F(v) + \langle \nabla F(v), x-v\rangle + \ldots + \tfrac{1}{(p-1)!}\nabla^{(p-1)} F(v)[x-v]^{p-1} + \tfrac{L}{(p-1)!}\|x - v\|^{p-1}(x - v). 
\end{equation}
A key observation is that there can be some $x$-subproblems that do not need to refer to $\lambda$; e.g., the one employed in Algorithm~\ref{Algorithm:DE}.  Indeed, we compute $x_{k+1} \in \XCal$ that approximately satisfies the following condition: 
\begin{equation*}
\langle F_{v_{k+1}}(x_{k+1}), x - x_{k+1} \rangle \geq 0 \textnormal{ for all } x \in \XCal. 
\end{equation*}
It suffices to return $x_{k+1} \in \XCal$ with a sufficiently good quality to give us $\lambda_{k+1} > 0$ using a simple update rule. Intuitively, such an adaptive strategy makes sense since $\lambda_{k+1}$ serves as the stepsize in the dual space and we need to be aggressive as the iterate $x_{k+1}$ approaches the set of optimal solutions to the VI. Meanwhile, the quantity $\|x_{k+1} - v_{k+1}\|$ can be used to measure the distance between $x_{k+1}$ and an optimal solution, and the order $p \in \{1, 2, 3,\ldots\}$ quantifies the relationship between the closeness and the exploitation of high-order derivative information. In summary, $\lambda_{k+1}$ becomes larger for a better iterate $x_{k+1} \in \XCal$ and such a choice leads to a faster global rate of convergence. 

\paragraph{Restart version of \textsf{Perseus}.} We summarize the restarted versions of our $p^{\textnormal{th}}$-order method in Algorithm~\ref{Algorithm:restart}. These methods, which we refer to as \textsf{Perseus-restart}($p$, $x_0$, $L$, $\sigma$, $D$, $T$,  $\textsf{opt}$), combine Algorithm~\ref{Algorithm:DE} with two restart strategies;~\citep[c.f.][]{Nemirovski-1985-Optimal,Nesterov-2013-Gradient,Donoghue-2015-Adaptive,Nesterov-2018-Lectures}. 

Restart schemes stop an algorithm when a criterion is satisfied and then restart the algorithm with a new input. Originally studied in the setting of momentum-based methods, restarting has been recognized as an important tool for designing linearly convergent algorithms when the objective function is strongly/uniformly convex~\citep{Nemirovski-1985-Optimal,Nesterov-2013-Gradient,Ghadimi-2013-Optimal} or has some other structures~\citep{Freund-2018-New,Necoara-2019-Linear,Renegar-2022-Simple}. Note that strong monotonicity is a generalization of such regularity conditions. As such, it is natural to consider a restarted version of our method, hoping to achieve linear convergence. Accordingly, at each iteration of Algorithm~\ref{Algorithm:restart}, we use $x_{k+1} = \textsf{Perseus}(p, x_k, L, t, \textsf{opt})$ as a subroutine. In other words, we simply restart \textsf{Perseus} every $t \geq 1$ iterations and take advantage of average iterates or best iterates to generate $x_{k+1}$ from $x_k$. In addition, it is worth mentioning that the choice of $t$ can be specialized to different settings and/or different type of convergence guarantees. In particular, we set $\textsf{opt} = 0$ for the uniformly monotone setting and $\textsf{opt} = 1$ for the strongly monotone setting. 

In the context of VI, the restarting strategies have been used to extend high-order extragradient methods~\citep{Bullins-2022-Higher,Adil-2022-Optimal} from the monotone setting to the strongly monotone setting~\citep{Ostroukhov-2020-Tensor,Huang-2022-Approximation}.  Several papers also focus on the investigation of adaptive restart schemes that speed up the convergence of classical first-order methods~\citep{Giselsson-2014-Monotonicity,Donoghue-2015-Adaptive} and provide theoretical guarantees in a general setting where the objective function is smooth and has H\"{o}lderian growth~\citep{Roulet-2017-Sharpness,Fercoq-2019-Adaptive}. A drawback of these schemes is that they rely on knowing appropriately accurate approximations of problem parameters.  The same issue arises for our method, given that Algorithm~\ref{Algorithm:restart} needs to choose $T_{\textnormal{inner}} \geq 1$. In the optimization setting, recent work by~\citet{Renegar-2022-Simple} shows how to alleviate this problem via a simple restart scheme that makes no attempt to learn parameter values and only requires the information that is readily available in practice.  It is an interesting open question as to whether such a scheme can be found in the VI setting for \textsf{Perseus}. 

\subsection{Main results}
We provide our main results on the convergence rate for Algorithm~\ref{Algorithm:DE} and~\ref{Algorithm:restart} in terms of the number of calls of the subproblem solvers. Note that Assumption~\ref{Assumption:smooth} will be made throughout and we impose the Minty condition (see Definition~\ref{def:Minty}) for the nonmonotone setting. 

\paragraph{Monotone setting.} The following theorems give us the global convergence rate of Algorithm~\ref{Algorithm:DE} and~\ref{Algorithm:restart} for smooth and (uniformly/strongly) monotone VIs. 
\begin{theorem}\label{Thm:monotone-global}
Suppose that Assumption~\ref{Assumption:smooth} holds and $F: \br^d \rightarrow \br^d$ is monotone and let $\epsilon \in (0, 1)$. The required number of iterations is 
\begin{equation*}
T = O\left(\left(\frac{LD^{p+1}}{\epsilon}\right)^{\frac{2}{p+1}}\right), 
\end{equation*}
where $\hat{x} = \textsf{Perseus}(p, x_0, L, T, 0)$ satisfies $\textsc{gap}(\hat{x}) \leq \epsilon$ and the total number of calls of the subproblem solvers is equal to $T$. Here, $p \in \{1, 2, \ldots\}$ is an order, $L > 0$ is a Lipschitz constant for $(p-1)^{\textnormal{th}}$-order smoothness of $F$ and $D > 0$ is the diameter of $\XCal$. 
\end{theorem}
\begin{theorem}\label{Thm:monotone-restart}
Suppose that Assumption~\ref{Assumption:smooth} holds and $F: \br^d \rightarrow \br^d$ is $p^{\textnormal{th}}$-order $\mu$-uniformly monotone and let $\epsilon \in (0, 1)$. The required number of iterations is 
\begin{equation*}
T = O\left(\log_2\left(\frac{D}{\epsilon}\right)\right), 
\end{equation*}
such that $\hat{x} = \textsf{Perseus-restart}(p, x_0, L, \mu, D, T, 0)$ satisfies $\|\hat{x} - x^\star\| \leq \epsilon$ and the total number of calls of the subproblem solvers is bounded by
\begin{equation*}
O\left(\kappa^{\frac{2}{p+1}}\log_2\left(\frac{D}{\epsilon}\right)\right), 
\end{equation*}
where $p \in \{1, 2, \ldots\}$ is an order, $\kappa = L/\mu > 0$ is the condition number of $F$, $D > 0$ is the diameter of $\XCal$ and $x^\star \in \XCal$ is one weak solution. 
\end{theorem}
\begin{remark}
For the first-order methods (i.e., the case of $p=1$), the convergence guarantee in Theorem~\ref{Thm:monotone-global} recovers the global rate of $O(L/\epsilon)$ in~\citet[Theorem~2]{Nesterov-2007-Dual}. The same rate has been derived for other first-order methods~\citep{Nemirovski-2004-Prox,Monteiro-2010-Complexity,Mokhtari-2020-Convergence,Kotsalis-2022-Simple} and is known to match the established lower bound~\citep{Ouyang-2021-Lower}.  For the second-order and high-order methods (i.e., the case of $p \geq 2$), our results improve upon the state-of-the-art results~\citep{Monteiro-2012-Iteration,Bullins-2022-Higher,Lin-2023-Monotone,Jiang-2022-Generalized} by shaving off an additional $\log\log(1/\epsilon)$ factor.  
\end{remark}
\begin{remark}
For the first-order methods, Theorem~\ref{Thm:monotone-restart} recovers the global linear convergence rate achieved by the dual extrapolation method and matches the lower bound~\citep{Zhang-2022-Lower}. For the second-order and high-order methods, our results are new and we believe that these bounds can not be further improved although we do not know of lower bounds.
\end{remark}
\paragraph{Local convergence.} We present the local convergence property of our methods for the strongly monotone VIs. 
\begin{theorem}\label{Thm:monotone-local}
Suppose that Assumption~\ref{Assumption:smooth} holds and $F: \br^d \rightarrow \br^d$ is $\mu$-strongly monotone and let $\{x_k\}_{k=0}^{T+1}$ be generated by $\textsf{Perseus-restart}(p, x_0, L, \mu, D, T, 1)$. Then, the following statement holds true, 
\begin{equation*}
\|x_{k+1} - x^\star\| \leq \sqrt{\tfrac{2^p(5p-2)\kappa}{p!}}\|x_k - x^\star\|^{\frac{p+1}{2}}, 
\end{equation*}
where $\kappa = L/\mu > 0$ is the condition number of the VI, $D > 0$ is the diameter of $\XCal$ and $x^\star$ is the unique weak solution of the VI. As a consequence,  if $p \geq 2$ and the following condition holds true, 
\begin{equation*}
\|x_0 - x^\star\| \leq \tfrac{1}{2}\left(\tfrac{p!}{2^p(5p-2)\kappa}\right)^{\frac{1}{p-1}},
\end{equation*}
the iterates $\{x_k\}_{k=0}^{T+1}$ converge to $x^\star \in \XCal$ in at least a superlinear rate. 
\end{theorem}
\begin{remark}
The local convergence guarantee in Theorem~\ref{Thm:monotone-local} is derived for the second-order and high-order methods (i.e., the case of $p \geq 2$) and is posited as their advantage over first-order method if we hope to pursue high-accuracy solutions. In this context,~\citet{Jiang-2022-Generalized} provided the same local convergence guarantee for the generalized optimistic gradient methods as our results in Theorem~\ref{Thm:monotone-local} but without counting the complexity bound of binary search procedure. 
\end{remark}
\paragraph{Nonmonotone setting.} We consider smooth and nonmonotone VIs satisfying the Minty condition and present the global rate of Algorithm~\ref{Algorithm:DE} and~\ref{Algorithm:restart} in terms of the number of calls of the subproblem solvers. 
\begin{theorem}\label{Thm:nonmonotone-global}
Suppose that Assumption~\ref{Assumption:smooth} and the Minty condition hold true and let $\epsilon \in (0, 1)$. The required number of iterations is 
\begin{equation*}
T = O\left(\left(\frac{LD^{p+1}}{\epsilon}\right)^{\frac{2}{p}}\right), 
\end{equation*}
such that $\hat{x} = \textsf{Perseus}(p, x_0, L, T, 2)$ satisfies $\textsc{res}(\hat{x}) \leq \epsilon$ and the total number of calls of the subproblem solvers is equal to $T$. Here, $p \in \{1, 2, \ldots\}$ is an order, $L > 0$ is the Lipschitz constant for $(p-1)^{\textnormal{th}}$-order smoothness of $F$ and $D > 0$ is the diameter of $\XCal$. 
\end{theorem}
\begin{remark}
The convergence guarantee in Theorem~\ref{Thm:nonmonotone-global} was derived for other first-order methods~\citep{Dang-2015-Convergence,Song-2020-Optimistic} for the case of $p = 1$.  They are completely new for the case of $p \geq 2$ in the literature to our knowledge. 
\end{remark}
\begin{remark}
For smooth and nonmonotone VIs, we obtain the same convergence rate as Theorem~\ref{Thm:monotone-restart} to a weak solution rather than a strong solution under $p^{\textnormal{th}}$-order $\mu$-uniform Minty condition. The same local superlinear rate as Theorem~\ref{Thm:monotone-local} to a weak solution can be obtained under $\mu$-strong Minty condition. The proof would be the same as that used for proving Theorem~\ref{Thm:monotone-restart} and~\ref{Thm:monotone-local}. 
\end{remark}
\paragraph{Lower bound.} We provide the lower bound for the monotone setting under a linear span assumption. Our analysis and hard instance are largely inspired by the constructions and techniques from~\citet{Nesterov-2021-Implementable} and~\citet{Adil-2022-Optimal}. However, different from~\citet{Adil-2022-Optimal}, our lower bound is established for a wide class of $p^{\textnormal{th}}$-order dual extrapolation methods that include our method, rather than $p^{\textnormal{th}}$-order methods restricted to solve the primal problem in~\citet[Eq.~(11)]{Adil-2022-Optimal}.  

For constructing the problems that are difficult for our $p^{\textnormal{th}}$-order methods, it is convenient to consider the saddle point problem, $\min_{z \in \ZCal} \max_{y \in \YCal} f(z, y)$, which is a special monotone VI defined via an operator $F$ and a closed, convex and bounded set $\XCal$ as follows: 
\begin{equation*}
x = \begin{bmatrix} z \\ y \end{bmatrix}, \qquad 
F(x) = \begin{bmatrix} \nabla_z f(z, y) \\ -\nabla_y f(z, y) \end{bmatrix}, \qquad 
\XCal = \ZCal \times \YCal. 
\end{equation*}
Let us describe the abilities of $p^{\textnormal{th}}$-order methods of degree $p \geq 2$ in generating the new iterates. In particular, the output of oracle at a point $\bar{x} \in \XCal$ consists in the set of multi-linear forms given by 
\begin{equation*}
F(\bar{x}), \nabla F(\bar{x}), \ldots, \nabla^{(p-1)} F(\bar{x}). 
\end{equation*}
To that end, we assume that the $p^{\textnormal{th}}$-order method in our algorithm class is able to generate a sequence of iterates $\{x_k\}_{k \geq 0}$ satisfying the recursive condition:
\begin{equation*}
\begin{array}{l}
s \in \textnormal{Lin}(F(x_0), \ldots, F(x_k)), \qquad \bar{x} = \argmax_{x \in \XCal} \{\langle s, x-x_0\rangle - \frac{1}{2}\|x-x_0\|^2\}, \\
x_{k+1} \in \XCal \textnormal{ satisfies that } \langle \Phi_{a, \gamma, \bar{x}}(x_{k+1} - x_k), x - x_{k+1} \rangle \geq 0 \textnormal{ for all } x \in \XCal, 
\end{array}
\end{equation*}
where $\Phi_{a, \gamma, \bar{x}}(h) = a_0 F(\bar{x}) + \sum_{i=1}^{p-1} a_i \nabla^{(i)} F(\bar{x})[h]^i + \gamma\|h\|^{p-1}h$ with $a \in \br^p$ and $\gamma > 0$. Our assumption about the form of $p^{\textnormal{th}}$-order methods in our algorithm class is summarized as follows: 
\begin{assumption}\label{Assumption:linear-span}
The $p^{\textnormal{th}}$-order method generates a sequence of iterates $\{x_k\}_{k \geq 0}$ satisfying the following recursive condition: for all $k \geq 0$, we have that $x_{k+1} \in \XCal$ satisfies that $\langle \Phi_{a, \gamma, \bar{x}}(x_{k+1} - x_k), x - x_{k+1} \rangle \geq 0$ for all $x \in \XCal$, where 
\begin{equation*}
\bar{x} = \argmax_{x \in \XCal} \{\langle s, x-x_0\rangle - \tfrac{1}{2}\|x-x_0\|^2\} \textnormal{ and } s \in \textnormal{Lin}(F(x_0), \ldots, F(x_k)).  
\end{equation*}
\end{assumption}
Assumption~\ref{Assumption:linear-span} is a generalization of a classical linear span assumption~\citep{Nesterov-2021-Implementable} and is well satisfied by various dual extrapolation methods, including Algorithm~\ref{Algorithm:DE}. However, it might not hold true for other VI methods, such as extragradient methods and their variants~\citep{Monteiro-2013-Accelerated,Bullins-2022-Higher,Huang-2022-Cubic,Jiang-2022-Generalized,Lin-2023-Monotone}. Removing Assumption~\ref{Assumption:linear-span} using the rotation techniques~\citep{Arjevani-2019-Oracle,Carmon-2020-Lower,Ouyang-2021-Lower} is a challenging task; indeed, due to the nonlinear coupling terms of our hard instances (see Subsection~\ref{subsec:proof_lowerbound}), the previous analysis cannot be directly applied. We leave further exploration of this topic to future work.

Note that the same lower bound has been recently established in~\citet{Adil-2022-Optimal} for a special class of $p^{\textnormal{th}}$-order methods restricted to solve the primal problem under Assumption~\ref{Assumption:linear-span}. Indeed, their construction is based on a saddle-point problem $\min_{z \in \ZCal} \max_{y \in \YCal} f(z, y)$ and they assume that any method in their algorithm class not only satisfies Assumption~\ref{Assumption:linear-span} but has the access to $\nabla \phi(\bar{z}), \ldots, \nabla^{(p)} \phi(\bar{z})$ where $\phi(z) = \max_{y \in \YCal} f(z, y)$ refers to the objective function of primal problem (see~\citet[Lemma~4.3]{Adil-2022-Optimal}). Proving the lower bound for general $p^{\textnormal{th}}$-order methods under Assumption~\ref{Assumption:linear-span} requires a new hard instance, which is a nonlinear generalization of the instance used in~\citet{Adil-2022-Optimal}. 

The following theorem summarizes our main result. The proof details can be found in Section~\ref{sec:proof}. 
\begin{theorem}\label{Thm:monotone-lower-bound}
Fixing $p \geq 2$, $L > 0$ and $T > 0$ and letting $d \geq 4T + 1$ be the problem dimension. There exists two closed, convex and bounded sets $\ZCal, \YCal \subseteq \br^d$ and a function $f(z, y): \ZCal \times \YCal \rightarrow \br$ that is convex-concave with an optimal saddle-point solution $(z_\star, y_\star) \in \ZCal \times \YCal$ such that the iterates $\{(z_k, y_k)\}_{k \geq 0}$ generated by any $p^{\textnormal{th}}$-order method under Assumption~\ref{Assumption:linear-span} must satisfy
\begin{equation*}
\min_{0 \leq k \leq T}\left\{\max_{y \in \YCal} f(z_k, y) - \min_{z \in \ZCal} f(z, y_k)\right\} \geq \left(\tfrac{1}{4^{p+2}(p+1)!}\right)LD_\ZCal D_\YCal^p T^{-\frac{p+1}{2}}. 
\end{equation*}
\end{theorem}
\begin{remark}
The lower bound in Theorem~\ref{Thm:monotone-lower-bound} shows that any $p^{\textnormal{th}}$-order method satisfying Assumption~\ref{Assumption:linear-span} requires at least $\Omega((LD_\ZCal D_\YCal^p)^{\frac{2}{p+1}}\epsilon^{-\frac{2}{p+1}})$ iterations to reach an $\epsilon$-weak solution. Combined this result with Theorem~\ref{Thm:monotone-global} shows that Algorithm~\ref{Algorithm:DE} is an optimal $p^\textnormal{th}$-order method for solving smooth and monotone VIs. As mentioned before, we have improved the results in~\citet{Adil-2022-Optimal} by constructing a new hard instance and deriving the same lower bound for a more broad class of $p^{\textnormal{th}}$-order methods that include both Algorithm~\ref{Algorithm:DE} and the high-order extragradient method in~\citet{Adil-2022-Optimal}. 
\end{remark}
\begin{remark}
For the lower bound for finding an $\epsilon$-strong solution in monotone setting, the case for first-order VI methods have been investigated in~\citet{Diakonikolas-2020-Halpern}. The key idea is to use the lower bound for finding an $\epsilon$-weak solution~\citep{Ouyang-2021-Lower} and the algorithmic reductions to derive lower bounds. However, such a reduction is mostly based on the high-order generalization of Halpern iteration and is thus beyond the scope of the current manuscript. In particular, we have developed a simple and optimal $p^{\textnormal{th}}$-order VI method for finding an $\epsilon$-weak solution in the monotone setting. However, the optimal algorithm for finding an $\epsilon$-strong solution in the monotone setting is likely to be different as evidenced by~\citet{Diakonikolas-2020-Halpern}. Computing an $\epsilon$-strong solution and/or an $\epsilon$-weak solution are complementary, yet different, and they indeed deserve separate study in their own right. Moreover, the lower bound for finding an $\epsilon$-strong solution under the Minty condition is largely unexplored and missing even in the current literature for first-order VI methods. 
\end{remark} 

\section{Convergence Analysis}\label{sec:proof}
We present the convergence analysis for our $p^{\textnormal{th}}$-order method (Algorithm~\ref{Algorithm:DE}) and its restarted version (Algorithm~\ref{Algorithm:restart}). In particular, we provide the global convergence guarantee (Theorems~\ref{Thm:monotone-global} and~\ref{Thm:monotone-restart}) and local convergence guarantee for the monotone setting (Theorems~\ref{Thm:monotone-local}). We analyze the nonmonotone setting under the Minty condition (Theorems~\ref{Thm:nonmonotone-global}). Finally, we establish the lower bound for the monotone setting under a linear span assumption (Theorem~\ref{Thm:monotone-lower-bound}). 

\subsection{Technical lemmas}
We define the following Lyapunov function for the iterates $\{x_k\}_{k \geq 0}$ that are generated by Algorithm~\ref{Algorithm:DE}: 
\begin{equation}\label{def:Lyapunov-discrete}
\ECal_k = \max_{v \in \XCal} \ \langle s_k, v - x_0\rangle - \tfrac{1}{2}\|v - x_0\|^2.
\end{equation}
This function is used to prove technical results that pertain to the dynamics of Algorithm~\ref{Algorithm:DE}.
\begin{lemma}\label{Lemma:DE-descent}
Suppose that Assumption~\ref{Assumption:smooth} holds true. For every integer $T \geq 1$, we have
\begin{equation*}
\sum_{k=1}^T \lambda_k \langle F(x_k), x_k - x\rangle \leq \ECal_0 - \ECal_T + \langle s_T, x - x_0\rangle -\tfrac{1}{10}\left(\sum_{k=1}^T \|x_k - v_k\|^2\right), \quad \textnormal{for all } x \in \XCal. 
\end{equation*}
\end{lemma}
\begin{proof}
By combining Eq.~\eqref{def:Lyapunov-discrete} and the definition of $v_{k+1}$, we have
\begin{equation*}
\ECal_k = \langle s_k, v_{k+1} - x_0\rangle - \tfrac{1}{2}\|v_{k+1} - x_0\|^2. 
\end{equation*}
Then, we have
\begin{eqnarray}\label{inequality:DE-descent-first}
\lefteqn{\ECal_{k+1} - \ECal_k = \langle s_{k+1}, v_{k+2} - x_0\rangle - \langle s_k, v_{k+1} - x_0\rangle - \tfrac{1}{2}\left(\|v_{k+2} - x_0\|^2 - \|v_{k+1} - x_0\|^2\right)} \\
& = & \langle s_{k+1} - s_k, v_{k+1} - x_0\rangle + \langle s_{k+1}, v_{k+2} - v_{k+1}\rangle - \tfrac{1}{2}\left(\|v_{k+2} - x_0\|^2 - \|v_{k+1} - x_0\|^2\right). \nonumber
\end{eqnarray}
By using the update formula for $v_{k+1}$ again, we have
\begin{equation*}
\langle x - v_{k+1}, s_k - v_{k+1} + x_0\rangle \leq 0, \quad \textnormal{for all } x \in \XCal. 
\end{equation*}
Letting $x = v_{k+2}$ in this inequality and using $\langle a, b\rangle = \frac{1}{2}(\|a+b\|^2 - \|a\|^2 - \|b\|^2)$, we have
\begin{equation}\label{inequality:DE-descent-second}
\langle s_k, v_{k+2} - v_{k+1}\rangle \leq \langle v_{k+1} - x_0, v_{k+2} - v_{k+1}\rangle = \tfrac{1}{2}\left(\|v_{k+2} - x_0\|^2 - \|v_{k+1} - x_0\|^2 - \|v_{k+2} - v_{k+1}\|^2\right). 
\end{equation}
Plugging Eq.~\eqref{inequality:DE-descent-second} into Eq.~\eqref{inequality:DE-descent-first} and using the update formula of $s_{k+1}$, we obtain:
\begin{eqnarray*}
\lefteqn{\ECal_{k+1} - \ECal_k \overset{\textnormal{Eq.~\eqref{inequality:DE-descent-second}}}{\leq} \langle s_{k+1} - s_k, v_{k+1} - x_0\rangle + \langle s_{k+1} - s_k, v_{k+2} - v_{k+1}\rangle - \tfrac{1}{2}\|v_{k+2} - v_{k+1}\|^2} \\
& = & \langle s_{k+1} - s_k, v_{k+2} - x_0\rangle - \tfrac{1}{2}\|v_{k+2} - v_{k+1}\|^2 \\
& \leq & \lambda_{k+1}\langle F(x_{k+1}), x_0 - v_{k+2}\rangle - \tfrac{1}{2}\|v_{k+2} - v_{k+1}\|^2 \\
& = & \lambda_{k+1} \langle F(x_{k+1}), x_0 - x\rangle + \lambda_{k+1} \langle F(x_{k+1}), x - x_{k+1}\rangle + \lambda_{k+1} \langle F(x_{k+1}), x_{k+1} - v_{k+2}\rangle - \tfrac{1}{2}\|v_{k+2} - v_{k+1}\|^2,
\end{eqnarray*}
for any $x \in \XCal$.
Summing up this inequality over $k = 0, 1, \ldots, T-1$ and changing the counter $k+1$ to $k$ yields that 
\begin{equation}\label{inequality:DE-descent-third}
\sum_{k=1}^T \lambda_k \langle F(x_k), x_k - x\rangle \leq \ECal_0 - \ECal_T + \underbrace{\sum_{k=1}^T \lambda_k \langle F(x_k), x_0 - x\rangle}_{\textbf{I}} + \underbrace{\sum_{k=1}^T \lambda_k \langle F(x_k), x_k - v_{k+1}\rangle - \tfrac{1}{2}\|v_k - v_{k+1}\|^2}_{\textbf{II}}. 
\end{equation}
Using the update formula for $s_{k+1}$ and letting $s_0 = 0_d \in \br^d$, we have
\begin{equation}\label{inequality:DE-descent-fourth}
\textbf{I} = \sum_{k=1}^T \langle \lambda_k F(x_k), x_0 - x\rangle = \sum_{k=1}^T \langle s_{k-1} - s_k, x_0 - x\rangle = \langle s_0 - s_T, x_0 - x\rangle = \langle s_T, x - x_0\rangle.
\end{equation}
Since $x_{k+1} \in \XCal$ satisfies Eq.~\eqref{condition:approximation}, we have 
\begin{equation}\label{opt:DE-main-update}
\langle F_{v_k}(x_k), x - x_k\rangle \geq -\tfrac{L}{p!}\|x_k - v_k\|^{p+1}, \quad \textnormal{for all } x \in \XCal,
\end{equation}
where $F_v(x): \br^d \rightarrow \br^d$ is defined for any fixed $v \in \XCal$ as follows:
\begin{equation*}
F_{v_k}(x) = F(v_k) + \langle \nabla F(v_k), x-v_k\rangle + \ldots + \tfrac{1}{(p-1)!}\nabla^{(p-1)} F(v_k)[x-v_k]^{p-1} + \tfrac{5L}{(p-1)!}\|x - v_k\|^{p-1}(x - v_k). 
\end{equation*}
Under Assumption~\ref{Assumption:smooth}, we obtain from~\citet[Fact 2.5]{Bullins-2022-Higher} or~\citet[Eq.~(6)]{Jiang-2022-Generalized} that 
\begin{equation}\label{opt:DE-main-smooth}
\|F(x_k) - F_{v_k}(x_k) + \tfrac{5L}{(p-1)!}\|x_k - v_k\|^{p-1}(x_k - v_k)\| \leq \tfrac{L}{p!}\|x_k - v_k\|^p. 
\end{equation}
Letting $x = v_{k+1}$ in Eq.~\eqref{opt:DE-main-update}, we have
\begin{equation}\label{inequality:DE-descent-fifth}
\langle F_{v_k}(x_k), x_k - v_{k+1}\rangle \leq \tfrac{L}{p!}\|x_k - v_k\|^{p+1}. 
\end{equation}
Inspired by Eq.~\eqref{opt:DE-main-smooth} and Eq.~\eqref{inequality:DE-descent-fifth}, we decompose $\langle F(x_k), x_k - v_{k+1}\rangle$ as follows:
\begin{eqnarray*}
\lefteqn{\langle F(x_k), x_k - v_{k+1}\rangle} \\
& = & \langle F(x_k) - F_{v_k}(x_k) + \tfrac{5L}{(p-1)!}\|x_k - v_k\|^{p-1}(x_k - v_k), x_k - v_{k+1}\rangle \\ 
& & + \langle F_{v_k}(x_k), x_k - v_{k+1}\rangle - \tfrac{5L}{(p-1)!}\|x_k - v_k\|^{p-1} \langle x_k - v_k, x_k - v_{k+1}\rangle \\
& \leq & \|F(x_k) - F_{v_k}(x_k) + \tfrac{5L}{(p-1)!}\|x_k - v_k\|^{p-1}(x_k - v_k)\| \cdot \|x_k - v_{k+1}\| \\
& & + \langle F_{v_k}(x_k), x_k - v_{k+1}\rangle - \tfrac{5L}{(p-1)!}\|x_k - v_k\|^{p-1} \langle x_k - v_k, x_k - v_{k+1}\rangle \\
& \overset{\textnormal{Eq.~\eqref{opt:DE-main-smooth} and Eq.~\eqref{inequality:DE-descent-fifth}}}{\leq} & \tfrac{L}{p!}\|x_k - v_k\|^p\|x_k - v_{k+1}\| + \tfrac{L}{p!}\|x_k - v_k\|^{p+1} - \tfrac{5L}{(p-1)!}\|x_k - v_k\|^{p-1} \langle x_k - v_k, x_k - v_{k+1}\rangle \\
& \leq & \tfrac{2L}{p!}\|x_k - v_k\|^{p+1} + \tfrac{L}{p!}\|x_k - v_k\|^p\|v_k - v_{k+1}\| - \tfrac{5L}{(p-1)!}\|x_k - v_k\|^{p-1} \langle x_k - v_k, x_k - v_{k+1}\rangle. 
\end{eqnarray*}
Note that we have
\begin{equation*}
\langle x_k - v_k, x_k - v_{k+1}\rangle = \|x_k - v_k\|^2 + \langle x_k - v_k, v_k - v_{k+1}\rangle \geq \|x_k - v_k\|^2 - \|x_k - v_k\|\|v_k - v_{k+1}\|. 
\end{equation*}
Putting these pieces together yields that 
\begin{equation*}
\langle F(x_k), x_k - v_{k+1}\rangle \leq \tfrac{(5p+1)L}{p!}\|x_k - v_k\|^p\|v_k - v_{k+1}\| - \tfrac{(5p-2)L}{p!}\|x_k - v_k\|^{p+1}. 
\end{equation*}
Since $\frac{1}{20p-8} \leq \frac{\lambda_k L \|x_k - v_k\|^{p-1}}{p!} \leq \frac{1}{10p+2}$ for all $k \geq 1$, we have
\begin{eqnarray}\label{inequality:DE-descent-sixth}
\textbf{II} & \leq & \sum_{k=1}^T \left(\tfrac{(5p+1)\lambda_k L}{p!}\|x_k - v_k\|^p\|v_k - v_{k+1}\| - \tfrac{1}{2}\|v_k - v_{k+1}\|^2 - \tfrac{(5p-2)\lambda_k L}{p!}\|x_k - v_k\|^{p+1}\right) \nonumber \\ 
& \leq & \sum_{k=1}^T \left(\tfrac{1}{2}\|x_k - v_k\|\|v_k - v_{k+1}\| - \tfrac{1}{2}\|v_k - v_{k+1}\|^2 - \tfrac{1}{4}\|x_k - v_k\|^2\right) \nonumber \\ 
& \leq & \sum_{k=1}^T \left(\max_{\eta \geq 0}\left\{\tfrac{1}{2}\|x_k - v_k\|\eta - \tfrac{1}{2}\eta^2\right\} - \tfrac{1}{4}\|x_k - v_k\|^2\right) \nonumber \\
& = & -  \tfrac{1}{8} \left(\sum_{k=1}^T \|x_k - v_k\|^2 \right). 
\end{eqnarray}
Plugging Eq.~\eqref{inequality:DE-descent-fourth} and Eq.~\eqref{inequality:DE-descent-sixth} into Eq.~\eqref{inequality:DE-descent-third} yields that 
\begin{equation*}
\sum_{k=1}^T \lambda_k \langle F(x_k), x_k - x\rangle \leq \ECal_0 - \ECal_T + \langle s_T, x - x_0\rangle -\tfrac{1}{8}\left(\sum_{k=1}^T \|x_k - v_k\|^2 \right). 
\end{equation*}
This completes the proof. 
\end{proof}
\begin{lemma}\label{Lemma:DE-error}
Suppose that Assumption~\ref{Assumption:smooth} and the Minty condition hold true and let $x \in \XCal$. For every integer $T \geq 1$, we have
\begin{equation*}
\sum_{k=1}^T \lambda_k \langle F(x_k), x_k - x\rangle \leq \tfrac{1}{2}\|x - x_0\|^2, \qquad \sum_{k=1}^T \|x_k - v_k\|^2 \leq 4\|x^\star - x_0\|^2, 
\end{equation*}
where $x^\star \in \XCal$ denotes the weak solution to the VI. 
\end{lemma}
\begin{proof}
For any $x \in \XCal$, we have
\begin{equation*}
\ECal_0 - \ECal_T + \langle s_T, x - x_0\rangle = \ECal_0 - \left(\max_{v \in \XCal} \ \langle s_T, v - x_0\rangle - \tfrac{1}{2}\|v - x_0\|^2\right) + \langle s_T, x - x_0\rangle. 
\end{equation*}
Since $s_0 = 0_d$, we have $\ECal_0 = 0$ and 
\begin{equation*}
\ECal_0 - \ECal_T + \langle s_T, x - x_0\rangle \leq - \left(\langle s_T, x - x_0\rangle - \tfrac{1}{2}\|x - x_0\|^2\right) + \langle s_T, x - x_0\rangle = \tfrac{1}{2}\|x - x_0\|^2. 
\end{equation*}
This together with Lemma~\ref{Lemma:DE-descent} yields that 
\begin{equation*}
\sum_{k=1}^T \lambda_k \langle F(x_k), x_k - x\rangle + \tfrac{1}{8}\left(\sum_{k=1}^T \|x_k - v_k\|^2\right) \leq \tfrac{1}{2}\|x - x_0\|^2, \quad \textnormal{for all } x \in \XCal, 
\end{equation*}
which implies the first inequality. Since the VI satisfies the Minty condition (see Definition~\ref{def:Minty}), there exists $x^\star \in \XCal$ such that $\langle F(x_k), x_k - x^\star\rangle \geq 0$ for all $k \geq 1$. Letting $x = x^\star$ in the above inequality yields the second inequality. 
\end{proof}
We provide a technical lemma establishing a lower bound for $\sum_{k=1}^T \lambda_k$. 
\begin{lemma}\label{Lemma:DE-control}
Suppose that Assumption~\ref{Assumption:smooth} and the Minty condition hold true. For every integer $k \geq 1$, we have
\begin{equation*}
\sum_{k=1}^T \lambda_k \geq \tfrac{p!}{(20p-8)L}\left(\tfrac{1}{4\|x^\star - x_0\|^2}\right)^{\frac{p-1}{2}} T^{\frac{p+1}{2}}, 
\end{equation*}
where $x^\star \in \XCal$ denotes the weak solution to the VI. 
\end{lemma}
\begin{proof}
Without loss of generality, we assume that $x_0 \neq x^\star$. For $p=1$, we have $\lambda_k = \frac{1}{12L}$ for all $k \geq 1$.  For $p \geq 2$, we have
\begin{equation*}
\sum_{k=1}^T (\lambda_k)^{-\frac{2}{p-1}}(\tfrac{p!}{(20p-8)L})^{\frac{2}{p-1}} \leq \sum_{k=1}^T (\lambda_k)^{-\frac{2}{p-1}}(\lambda_k\|x_k - v_k\|^{p-1})^{\frac{2}{p-1}} = \sum_{k=1}^T \|x_k - v_k\|^2  \overset{\text{Lemma}~\ref{Lemma:DE-error}}{\leq} 4\|x^\star - x_0\|^2.  
\end{equation*}
By the H\"{o}lder inequality, we have
\begin{equation*}
\sum_{k=1}^T 1 = \sum_{k=1}^T \left((\lambda_k)^{-\frac{2}{p-1}}\right)^{\frac{p-1}{p+1}} (\lambda_k)^{\frac{2}{p+1}} \leq \left(\sum_{k=1}^T (\lambda_k)^{-\frac{2}{p-1}}\right)^{\frac{p-1}{p+1}}\left(\sum_{k=1}^T \lambda_k\right)^{\frac{2}{p+1}}.
\end{equation*}
Putting these pieces together yields that 
\begin{equation*}
T \leq (4\|x^\star - x_0\|^2)^{\frac{p-1}{p+1}}\left(\tfrac{(20p-8)L}{p!}\right)^{\frac{2}{p+1}}\left(\sum_{k=1}^T \lambda_k\right)^{\frac{2}{p+1}},
\end{equation*}
Plugging this into the above inequality yields that 
\begin{equation*}
\sum_{k=1}^T \lambda_k \geq \tfrac{p!}{(20p-8)L}\left(\tfrac{1}{4\|x^\star - x_0\|^2}\right)^{\frac{p-1}{2}} T^{\frac{p+1}{2}}.   
\end{equation*}
This completes the proof. 
\end{proof}

\subsection{Proof of Theorem~\ref{Thm:monotone-global}}
We see from~\citet[Theorem~3.1]{Harker-1990-Finite} that at least one strong solution to the VI exists since $F$ is continuous and $\XCal$ is convex, closed and bounded. Since any strong solution is a weak solution if $F$ is further assumed to be monotone, we obtain that the VI satisfies the Minty condition. 

Letting $x \in \XCal$, we derive from the monotonicity of $F$ and the definition of $\tilde{x}_T$ (i.e., $\textsf{opt} = 0$) that
\begin{equation*}
\langle F(x), \tilde{x}_T - x\rangle = \tfrac{1}{\sum_{k=1}^T \lambda_k}\left(\sum_{k=1}^T \lambda_k \langle F(x), x_k - x\rangle\right) \leq \tfrac{1}{\sum_{k=1}^T \lambda_k}\left(\sum_{k=1}^T \lambda_k \langle F(x_k), x_k - x\rangle\right). 
\end{equation*}
Combining this inequality with the first inequality in Lemma~\ref{Lemma:DE-error} yields that
\begin{equation*}
\langle F(x), \tilde{x}_T - x\rangle \leq \tfrac{\|x - x_0\|^2}{2(\sum_{k=1}^T \lambda_k)}, \quad \textnormal{for all } x \in \XCal. 
\end{equation*}
Since $x_0 \in \XCal$, we have $\|x - x_0\| \leq D$ and hence
\begin{equation*}
\langle F(x), \tilde{x}_T - x\rangle \leq \tfrac{D^2}{2(\sum_{k=1}^T \lambda_k)}, \quad \textnormal{for all } x \in \XCal. 
\end{equation*}
Then, we combine Lemma~\ref{Lemma:DE-control} and the fact that $\|x^\star - x_0\| \leq D$ to obtain that 
\begin{equation*}
\langle F(x), \tilde{x}_T - x\rangle \leq \tfrac{2^p(5p-2)}{p!}LD^{p+1} T^{-\frac{p+1}{2}}, \quad \textnormal{for all } x \in \XCal. 
\end{equation*}
By the definition of a gap function (see Eq.~\eqref{eq:cc-gap}), we have
\begin{equation}\label{inequality:monotone-global-main}
\textsc{gap}(\tilde{x}_T) = \sup_{x \in \XCal} \ \langle F(x), \tilde{x}_T - x\rangle \leq \tfrac{2^p(5p-2)}{p!}LD^{p+1} T^{-\frac{p+1}{2}}. 
\end{equation}
Therefore, we conclude from Eq.~\eqref{inequality:monotone-global-main} that we can set 
\begin{equation*}
T = O\left(\left(\frac{LD^{p+1}}{\epsilon}\right)^{\frac{2}{p+1}}\right), 
\end{equation*}
such that $\hat{x} = \textsf{Perseus}(p, x_0, L, T, 0)$ satisfies $\textsc{gap}(\hat{x}) \leq \epsilon$. The total number of calls of the subproblem solvers is equal to $T$ since our algorithm calls the subproblem solvers once at each iteration. This completes the proof. 

\subsection{Proof of Theorem~\ref{Thm:monotone-restart}}
In the uniformly monotone setting with a convex, closed and bounded set, the solution $x^\star \in \XCal$ to the VI exists and is unique~\citep{Facchinei-2007-Finite} and the VI satisfies the Minty condition. 

We first consider the relationship between $\|\hat{x} - x^\star\|$ and $\|x_0 - x^\star\|$ where $\hat{x} = \textsf{Perseus}(p, x_0, L, T_{\textnormal{inner}}, 0)$. We derive from Jensen's inequality and the definition of $\tilde{x}_{T_{\textnormal{inner}}}$ that
\begin{equation*}
\|\tilde{x}_{T_{\textnormal{inner}}} - x^\star\|^{p+1} = \left\|\tfrac{1}{\sum_{k=1}^{T_{\textnormal{inner}}} \lambda_k}\left(\sum_{k=1}^{T_{\textnormal{inner}}} \lambda_k x_k\right) - x^\star\right\|^{p+1} \leq \tfrac{1}{\sum_{k=1}^{T_{\textnormal{inner}}} \lambda_k}\left(\sum_{k=1}^{T_{\textnormal{inner}}} \lambda_k\|x_k - x^\star\|^{p+1}\right). 
\end{equation*}
Since $F$ is $p^{\textnormal{th}}$-order $\mu$-uniformly monotone, we have
\begin{equation*}
\|x_k - x^\star\|^{p+1} \leq \tfrac{1}{\mu}\langle F(x_k) - F(x^\star), x_k - x^\star\rangle \leq \tfrac{1}{\mu}\langle F(x_k), x_k - x^\star\rangle. 
\end{equation*}
Putting these pieces together yields that 
\begin{equation}\label{inequality:monotone-restart-first}
\|\tilde{x}_{T_{\textnormal{inner}}} - x^\star\|^{p+1} \leq \tfrac{1}{\mu(\sum_{k=1}^{T_{\textnormal{inner}}} \lambda_k)}\left(\sum_{k=1}^{T_{\textnormal{inner}}} \lambda_k\langle F(x_k), x_k - x^\star\rangle\right). 
\end{equation}
Combining the first inequality in Lemma~\ref{Lemma:DE-error} with Eq.~\eqref{inequality:monotone-restart-first} yields that 
\begin{equation*}
\|\tilde{x}_{T_{\textnormal{inner}}} - x^\star\|^{p+1} \leq \tfrac{1}{2\mu(\sum_{k=1}^{T_{\textnormal{inner}}} \lambda_k)}\|x_0 - x^\star\|^2. 
\end{equation*}
This together with Lemma~\ref{Lemma:DE-control} and the fact that $\hat{x} = \tilde{x}_{T_{\textnormal{inner}}}$ yields that 
\begin{eqnarray}\label{inequality:monotone-restart-second}
\|\hat{x} - x^\star\|^{p+1} & \leq & \left((4\|x_0 - x^\star\|^2)^{\frac{p-1}{2}}\tfrac{10p-4}{p!}\tfrac{L}{\mu}t^{-\frac{p+1}{2}}\right)\|x_0 - x^\star\|^2 \\ 
& = & \left(\tfrac{2^p(5p-2)}{p!}\tfrac{L}{\mu}t^{-\frac{p+1}{2}}\right)\|x_0 - x^\star\|^{p+1}.  \nonumber
\end{eqnarray}
Since $x_{k+1} = \textsf{Perseus}(p, x_k, L, T_{\textnormal{inner}}, 0)$ in the scheme of Algorithm~\ref{Algorithm:restart} and 
\begin{equation}\label{inequality:monotone-restart-third}
T_{\textnormal{inner}} = \left\lceil \left(\tfrac{2^{p+1}(5p-2)}{p!}\tfrac{L}{\mu}\right)^{\frac{2}{p+1}}\right\rceil, 
\end{equation}
we have
\begin{equation}\label{inequality:monotone-restart-fourth}
\|x_{k+1} - x^\star\|^{p+1} \leq \tfrac{1}{2}\|x_k - x^\star\|^{p+1}, \quad \textnormal{for all } k = 0, 1, 2, \ldots, T. 
\end{equation}
Therefore, we conclude from Eq.~\eqref{inequality:monotone-restart-third} and Eq.~\eqref{inequality:monotone-restart-fourth} that we can set 
\begin{equation*}
T = O\left(\log_2\left(\frac{D}{\epsilon}\right)\right), 
\end{equation*}
such that $\hat{x} = \textsf{Perseus-restart}(p, x_0, L, \mu, T, 0)$ satisfies $\|\hat{x} - x^\star\| \leq \epsilon$. The total number of calls of the subproblem solvers is bounded by
\begin{equation*}
O\left(\left(\frac{L}{\mu}\right)^{\frac{2}{p+1}}\log_2\left(\frac{D}{\epsilon}\right)\right). 
\end{equation*}
This completes the proof. 

\subsection{Proof of Theorem~\ref{Thm:monotone-local}}
In the strongly monotone setting, the solution $x^\star \in \XCal$ to the VI exists and is unique~\citep{Facchinei-2007-Finite} and the VI satisfies the Minty condition. 

We first consider the relationship between $\|\hat{x} - x^\star\|$ and $\|x_0 - x^\star\|$ where $\hat{x} = \textsf{Perseus}(p, x_0, L, T_{\textnormal{inner}}, 1)$. Since $F$ is $\mu$-strongly monotone, we apply the same argument from the proof of Theorem~\ref{Thm:monotone-restart} and obtain that 
\begin{equation*}
\|\tilde{x}_{T_{\textnormal{inner}}} - x^\star\|^2 \leq \tfrac{1}{\mu(\sum_{k=1}^{T_{\textnormal{inner}}} \lambda_k)}\left(\sum_{k=1}^{T_{\textnormal{inner}}} \lambda_k\langle F(x_k), x_k - x^\star\rangle\right) \overset{\textnormal{Lemma~\ref{Lemma:DE-error}}}{\leq} \tfrac{1}{2\mu(\sum_{k=1}^{T_{\textnormal{inner}}} \lambda_k)}\|x_0 - x^\star\|^2. 
\end{equation*}
This together with Lemma~\ref{Lemma:DE-control} and the fact that $\hat{x} = \tilde{x}_{T_{\textnormal{inner}}}$ yields that 
\begin{equation*}
\|\hat{x} - x^\star\|^2 \leq \left(\tfrac{2^p(5p-2)}{p!}\tfrac{L}{\mu}t^{-\frac{p+1}{2}}\right)\|x_0 - x^\star\|^{p+1}. 
\end{equation*}
Since $x_{k+1} = \textsf{Perseus}(p, x_k, L, T_{\textnormal{inner}}, 1)$ in Algorithm~\ref{Algorithm:restart} and $T_{\textnormal{inner}} = 1$, we have
\begin{equation*}
\|x_{k+1} - x^\star\|^2 \leq \left(\tfrac{2^p(5p-2)\kappa}{p!}\right)\|x_k - x^\star\|^{p+1}, 
\end{equation*}
which implies that 
\begin{equation*}
\|x_{k+1} - x^\star\| \leq \sqrt{\tfrac{2^p(5p-2)\kappa}{p!}}\|x_k - x^\star\|^{\frac{p+1}{2}}. 
\end{equation*}
For the case of $p \geq 2$, we have $\frac{p+1}{2} \geq \frac{3}{2}$ and $p-1 \geq 1$. If the following condition holds true, 
\begin{equation*}
\|x_0 - x^\star\| \leq \tfrac{1}{2}\left(\tfrac{p!}{2^p(5p-2)\kappa}\right)^{\frac{1}{p-1}},
\end{equation*}
we have
\begin{eqnarray*}
\lefteqn{\left(\tfrac{2^p(5p-2)\kappa}{p!}\right)^{\frac{1}{p-1}}\|x_{k+1} - x^\star\| \leq \left(\tfrac{2^p(5p-2)\kappa}{p!}\right)^{\frac{p+1}{2(p-1)}}\|x_k - x^\star\|^{\frac{p+1}{2}}} \\
& = & \left(\left(\tfrac{2^p(5p-2)\kappa}{p!}\right)^{\frac{1}{p-1}}\|x_k - x^\star\|\right)^{\frac{p+1}{2}} \leq \left(\left(\tfrac{2^p(5p-2)\kappa}{p!}\right)^{\frac{1}{p-1}}\|x_0 - x^\star\|\right)^{(\frac{p+1}{2})^{k+1}} \leq (\tfrac{1}{2})^{(\frac{p+1}{2})^{k+1}}. 
\end{eqnarray*}
This completes the proof. 

\subsection{Proof of Theorem~\ref{Thm:nonmonotone-global}}
We see from the second inequality in Lemma~\ref{Lemma:DE-error} that
\begin{equation*}
\min_{1 \leq k \leq T} \|x_k - v_k\|^2 \leq \tfrac{1}{T}\sum_{k=1}^T \|x_k - v_k\|^2 \leq \tfrac{4\|x^\star - x_0\|^2}{T}.  
\end{equation*}
By the definition of $x_{k_T}$ (i.e., $\textsf{opt} = 2$), we have
\begin{equation}\label{inequality:nonmonotone-global-first}
\|x_{k_T} - v_{k_T}\|^2 \leq \tfrac{4\|x^\star - x_0\|^2}{T}.  
\end{equation}
Recalling that $x_{k+1} \in \XCal$ satisfies Eq.~\eqref{condition:approximation}, we have 
\begin{equation*}
\langle F_{v_k}(x_k), x - x_k\rangle \geq -\tfrac{L}{p!}\|x_k - v_k\|^{p+1},\quad \textnormal{for all } x \in \XCal, 
\end{equation*}
where $F_v(x): \br^d \rightarrow \br^d$ is defined for any fixed $v \in \XCal$ as follows:
\begin{equation*}
F_{v_k}(x) = F(v_k) + \langle \nabla F(v_k), x-v_k\rangle + \ldots + \tfrac{1}{(p-1)!}\nabla^{(p-1)} F(v_k)[x-v_k]^{p-1} + \tfrac{5L}{(p-1)!}\|x - v_k\|^{p-1}(x - v_k). 
\end{equation*}
Under Assumption~\ref{Assumption:smooth}, we have Eq.~\eqref{opt:DE-main-smooth} which further leads to
\begin{equation*}
\|F(x_k) - F_{v_k}(x_k)\| \leq \tfrac{(5p+1)L}{p!}\|x_k - v_k\|^p. 
\end{equation*}
Putting these pieces together yields that 
\begin{eqnarray*}
\lefteqn{\langle F(x_k), x_k - x\rangle = \langle F(x_k) - F_{v_k}(x_k), x_k - x\rangle + \langle F_{v_k}(x_k), x_k - x\rangle} \\
& \leq & \|F(x_k) - F_{v_k}(x_k)\|\|x_k - x\| + \tfrac{L}{p!}\|x_k - v_k\|^{p+1} \\
& \leq & \tfrac{L}{p!}\|x_k - v_k\|^p\left((5p+1)\|x_k - x\| + \|x_k - v_k\|\right),\quad \textnormal{for all } x \in \XCal. 
\end{eqnarray*}
This implies that (for all $x \in \XCal$) 
\begin{equation}\label{inequality:nonmonotone-global-second}
\langle F(x_k), x_k - x\rangle \leq \tfrac{(5p+1)L}{p!}\|x_k - v_k\|^p\|x_k - x\| + \tfrac{L}{p!}\|x_k - v_k\|^{p+1}. 
\end{equation}
Then, we derive from the fact that $\|x_k - x\| \leq D$ and $\|x_k - v_k\| \leq D$ that
\begin{equation*}
\langle F(x_k), x_k - x\rangle \leq \tfrac{(5p+2)LD}{p!}\|x_k - v_k\|^p, \quad \textnormal{for all } x \in \XCal. 
\end{equation*}
By the definition of a residue function (see Eq.~\eqref{eq:cc-residue}), we have
\begin{eqnarray*}
\textsc{res}(x_{k_T}) & = & \sup_{x \in \XCal} \ \langle F(x_{k_T}), x_{k_T} - x\rangle \leq \tfrac{(5p+2)LD}{p!}\|x_{k_T} - v_{k_T}\|^p \\ 
& \overset{\textnormal{Eq.~\eqref{inequality:nonmonotone-global-first}}}{\leq} & \tfrac{(5p+2)LD}{p!}\left(\tfrac{4\|x^\star - x_0\|^2}{T}\right)^{\frac{p}{2}}. 
\end{eqnarray*}
Since $x_0, x^\star \in \XCal$, we have $\|x^\star - x_0\| \leq D$ and hence
\begin{equation}\label{inequality:nonmonotone-global-third}
\textsc{res}(x_{k_T}) \leq \tfrac{2^p(5p+2)}{p!}LD^{p+1}T^{-\frac{p}{2}}. 
\end{equation}
Therefore, we conclude from Eq.~\eqref{inequality:nonmonotone-global-third}that we can set 
\begin{equation*}
T = O\left(\left(\frac{LD^{p+1}}{\epsilon}\right)^{\frac{2}{p}}\right), 
\end{equation*}
such that $\hat{x} = \textsf{Perseus}(p, x_0, L, T, 1)$ satisfies $\textsc{res}(\hat{x}) \leq \epsilon$. The total number of calls of the subproblem solvers is equal to $T$ since our algorithm calls the subproblem solvers once at each iteration. This completes the proof. 

\subsection{Proof of Theorem~\ref{Thm:monotone-lower-bound}}\label{subsec:proof_lowerbound}
We first construct a hard function instance for any $p^\textnormal{th}$-order method that satisfies Assumption~\ref{Assumption:linear-span}. The basic function that we will use is as follows:
\begin{equation*}
\eta(z, y) = \tfrac{1}{p} \sum_{i=1}^d (z^{(i)})^p \cdot y^{(i)}, \quad z \in \br_+^d, \ y \in \br_+^d. 
\end{equation*}
Fixing $(z, y) \in \br_+^d \times \br_+^d$, $(h_1, h_2) \in \br^d \times \br^d$ and $1 \leq m + n \leq p$, we have
\begin{equation}\label{Func:gradient-form}
\nabla^{(m, n)} \eta(z, y)[h_1]^m [h_2]^n = \tfrac{(p-1)!}{(p-m)!} \cdot \begin{cases}
\sum_{i=1}^d (z^{(i)})^{p-m} y^{(i)} (h_1^{(i)})^m, & \textnormal{if } n = 0, \\
\sum_{i=1}^d (z^{(i)})^{p-m} (h_1^{(i)})^m h_2^{(i)}, & \textnormal{if } n = 1, \\
0, & \textnormal{otherwise}. 
\end{cases}
\end{equation}
Note that $T \geq 1$ is an integer-valued parameter and $d \geq 4T+1$. We now define the following $4T \times 4T$ triangular matrix with two nonzero diagonals~\citep{Nesterov-2021-Implementable}: 
\begin{equation*}
U = \begin{bmatrix} 1 & -1 & 0 & \cdots & 0 \\ 0 & 1 & -1 & \cdots & 0 \\ & \cdots & \cdots & \cdots & \\ 0 & 0 & \cdots & 1 & -1 \\ 0 & 0 & \cdots & 0 & 1 \end{bmatrix},  \quad
U^{-1} = \begin{bmatrix} 1 & 1 & 1 & \cdots & 1 \\ 0 & 1 & 1 & \cdots & 1 \\ & \cdots & \cdots & \cdots & \\ 0 & 0 & \cdots & 1 & 1 \\ 0 & 0 & \cdots & 0 & 1 \end{bmatrix}, \quad
U^\top = \begin{bmatrix} 1 & 0 & \cdots & 0 & 0 \\ -1 & 1 & \cdots & 0 & 0 \\ & \cdots & \cdots & \cdots & \\ 0 & 0 & \cdots & 1 & 0 \\ 0 & 0 & \cdots & -1 & 1 \end{bmatrix}. 
\end{equation*}
Now, we introduce $d \times d$ upper triangular matrix $A$ with the following structure: 
\begin{equation*}
A = \begin{bmatrix} U & 0 \\ 0 & I_{d-4T} \end{bmatrix}. 
\end{equation*}
We are now ready to characterize a novel hard function and the corresponding two constraint sets: 
\begin{align*}
& f(z, y) = \tfrac{L}{2^{p+1} p!}\left(\eta(Az, y) - \tfrac{1}{p(p+1)} \sum_{i=2}^{4T} (y^{(i)})^{p+1} - (z^{(1)} - 4T + \tfrac{1}{p}) \cdot y^{(1)}\right), \\
& \ZCal = \left\{z \in \br^d: 
\begin{array}{c}
0 \leq z^{(i)} \leq 4T-i+1 \textnormal{ and } z^{(i+1)} \leq z^{(i)} \textnormal{ for all } 1 \leq i \leq 4T \\ 
\textnormal{ and } z^{(i)} = 0 \textnormal{ for all } i > 4T 
\end{array}\right\}, \\
& \YCal = \{y \in \br^d: 0 \leq y^{(i)} \leq 1 \textnormal{ for all } 1 \leq i \leq 4T \textnormal{ and } y^{(i)} = 0 \textnormal{ for all } i > 4T\}. 
\end{align*}
If $m+n=p$, Eq.~\eqref{Func:gradient-form} implies that the only nonzero $(m, n)^{\textnormal{th}}$-order derivatives of $f$ are $\nabla^{(0, p)} f(z, y)$, $\nabla^{(p-1, 1)} f(z, y)$ and $\nabla^{(p, 0)} f(z, y)$. It is also clear that the function $f: \ZCal \times \YCal \rightarrow \br$ is convex in $z$ and concave in $y$. This implies that the computation of an optimal saddle-point solution of $f(z, y)$ is equivalent to solving a monotone VI with $x = \begin{bmatrix} z \\ y \end{bmatrix}$ and
\begin{equation*}
F(x) = \begin{bmatrix} \nabla^{(1, 0)} f(z, y) \\ - \nabla^{(0, 1)} f(z, y) \end{bmatrix} = \tfrac{L}{2^{p+1} p!} \cdot \begin{bmatrix} A^\top \nabla^{(1, 0)} \eta(Az, y) - y^{(1)} \cdot e_d^{(1)} \\ -\nabla^{(0, 1)} \eta(Az, y) + \tfrac{1}{p} \sum_{i=2}^{4T} (y^{(i)})^p \cdot e_d^{(i)} + (z^{(1)} - 4T + \tfrac{1}{p}) \cdot e_d^{(1)} \end{bmatrix}. 
\end{equation*}

\paragraph{Step 1.} We show that $F: \ZCal \times \YCal \rightarrow \br^{2d}$ is $(p-1)^{\textnormal{th}}$-order smooth with a Lipschitz constant $L > 0$. Indeed, we have
\begin{eqnarray}\label{Func:smooth}
\lefteqn{\|\nabla^{(p-1)} F(x) - \nabla^{(p-1)} F(x')\|_\op \leq \|\nabla^{(0, p)} f(z, y) - \nabla^{(0, p)} f(z', y')\|_\op} \\
& & + p\|\nabla^{(p-1, 1)} f(z, y) - \nabla^{(p-1, 1)} f(z', y')\|_\op + \|\nabla^{(p, 0)} f(z, y) - \nabla^{(p, 0)} f(z', y')\|_\op.  \nonumber
\end{eqnarray}
We let $h = (h_1, h_2) \in \br^d \times \br^d$ and have
\begin{equation*}
\nabla^{(0, p)} f(z, y)[h_2]^p = - \tfrac{L}{2^{p+1} p} \cdot \sum_{i=2}^{4T} (y^{(i)})(h_2^{(i)})^p, 
\end{equation*}
and 
\begin{equation*}
\nabla^{(p-1, 1)} f(z, y)[h_1]^{p-1} [h_2] = \begin{cases}
\tfrac{L}{16} \cdot (\nabla^{(1, 1)} \eta(Az, y)[Ah_1][h_2] - h_1^{(1)}h_2^{(1)}), & \textnormal{if } p = 2, \\ 
\tfrac{L}{2^{p+1} p!} \cdot (\nabla^{(p-1, 1)} \eta(Az, y)[Ah_1]^{p-1} [h_2]),  & \textnormal{otherwise}, 
\end{cases}
\end{equation*}
and
\begin{equation*}
\nabla^{(p, 0)} f(z, y)[h_1]^p = \tfrac{L}{2^{p+1} p!} \cdot \nabla^{(p, 0)} \eta(Az, y)[Ah_1]^p.  
\end{equation*}
By the Cauchy-Schwartz inequality and $\|A\| \leq 2$ (see~\citet[Eq. (4.2)]{Nesterov-2021-Implementable}), we have 
\begin{eqnarray*}
\lefteqn{\|\nabla^{(0, p)} f(z, y) - \nabla^{(0, p)} f(z', y')\|_\op} \\
& \leq & \sup_{\|h\| = 1} \left\{ \|\nabla^{(0, p)} f(z, y)[h_2]^p - \nabla^{(0, p)} f(z', y')[h_2]^p\|\right\} \\
& \leq & \sup_{\|h\| = 1} \left\{\tfrac{L}{2^{p+1} p} \cdot \|y - y'\|\|h_2\|^p\right\} \leq \tfrac{L}{2^{p+1} p} \cdot \|x - x'\| \leq \tfrac{L}{16} \cdot \|x - x'\|, 
\end{eqnarray*}
and 
\begin{eqnarray*}
\lefteqn{\|\nabla^{(p-1, 1)} f(z, y) - \nabla^{(p-1, 1)} f(z', y')\|_\op} \\
& \leq & \sup_{\|h\| = 1} \left\{ \|\nabla^{(p-1, 1)} f(z, y)[h_1]^{p-1}[h_2] - \nabla^{(p-1, 1)} f(z', y')[h_1]^{p-1}[h_2]\|\right\} \\
& \leq & \sup_{\|h\| = 1} \left\{ \tfrac{L}{2^{p+1} p!} \cdot \left(\|\nabla^{(p-1, 1)} \eta(Az, y)[Ah_1]^{p-1}[h_2] - \nabla^{(p-1, 1)} \eta(Az', y')[Ah_1]^{p-1}[h_2]\|\right)\right\} \\
& \overset{\textnormal{Eq.~\eqref{Func:gradient-form} and } \|A\| \leq 2}{\leq} & \sup_{\|h\| = 1} \left\{\tfrac{L}{2^{p+1} p} \cdot 2^p \cdot \|z - z'\|\|h_1\|^{p-1}\|h_2\|\right\} \leq \tfrac{L}{2p} \cdot \|x - x'\|, 
\end{eqnarray*}
and 
\begin{eqnarray*}
\lefteqn{\|\nabla^{(p, 0)} f(z, y) - \nabla^{(p, 0)} f(z', y')\|_\op} \\
& \leq & \sup_{\|h\| = 1} \left\{ \|\nabla^{(p, 0)} f(z, y)[h_1]^p - \nabla^{(p, 0)} f(z', y')[h_1]^p\|\right\} \\
& \leq & \sup_{\|h\| = 1} \left\{ \tfrac{L}{2^{p+1} p!} \cdot \left(\|\nabla^{(p, 0)} \eta(Az, y)[Ah_1]^p - \nabla^{(p, 0)} \eta(Az', y')[Ah_1]^p\|\right)\right\} \\
& \overset{\textnormal{Eq.~\eqref{Func:gradient-form} and } \|A\| \leq 2}{\leq} & \sup_{\|h\| = 1} \left\{\tfrac{L}{2^{p+1} p} \cdot 2^p \cdot \|y - y'\|\|h_1\|^p\right\} \leq \tfrac{L}{2p} \cdot \|x - x'\| \leq \tfrac{L}{4} \cdot \|x - x'\|. 
\end{eqnarray*}
Plugging the above equation into Eq.~\eqref{Func:smooth} yields the desired result. 

\paragraph{Step 2.} We show that there exists an optimal solution $x_\star = (z_\star, y_\star) \in \ZCal \times \YCal$ such that $F(x_\star) = \textbf{0}_{2d}$ and compute the optimal value of $f(z_\star, y_\star)$. By the definition, we have $F(x_\star) = \textbf{0}_{2d}$ is equivalent to the following statement: 
\begin{equation}\label{Func:opt}
\begin{cases}
A^\top \nabla^{(1, 0)} \eta(Az_\star, y_\star) - y_\star^{(1)} \cdot e_d^{(1)} = \textbf{0}_d, \\
\nabla^{(0, 1)} \eta(Az_\star, y_\star) - \tfrac{1}{p} \sum_{i=2}^{4T} (y_\star^{(i)})^p \cdot e_d^{(i)} - (z_\star^{(1)} - 4T + \tfrac{1}{p}) \cdot e_d^{(1)} = \textbf{0}_d.
\end{cases}
\end{equation}
Note that 
\begin{align*}
\nabla^{(1, 0)} \eta(Az_\star, y_\star) & = \sum_{i=1}^d ((Az_\star)^{(i)})^{p-1} y_\star^{(i)} e_d^{(i)}, \\
\nabla^{(0, 1)} \eta(Az_\star, y_\star) & = \tfrac{1}{p}\left(\sum_{i=1}^d ((Az_\star^{(i)}))^p e_d^{(i)}\right). 
\end{align*}
We claim that an optimal solution $x_\star = (z_\star, y_\star)$ is given by 
\begin{equation}\label{Sol:opt}
z_\star^{(i)} = \begin{cases} 4T - i + 1, & \textnormal{if } 1 \leq i \leq 4T, \\ 0 & \textnormal{otherwise}. \end{cases} \qquad 
y_\star^{(i)} = \begin{cases} 1, & \textnormal{if } 1 \leq i \leq 4T, \\ 0 & \textnormal{otherwise}. \end{cases}
\end{equation}
Indeed, we can see from the definition of $\ZCal \times \YCal$ that $(z_\star, y_\star) \in \ZCal \times \YCal$ and the definition of $A$ that 
\begin{equation*}
(Az_\star)^{(i)} = \begin{cases} 1, & \textnormal{if } 1 \leq i \leq 4T, \\ 0 & \textnormal{otherwise}. \end{cases}
\end{equation*}
This implies that 
\begin{equation*}
\nabla^{(1, 0)} \eta(Az_\star, y_\star) = \sum_{i=1}^{4T} e_d^{(i)}, \quad \nabla^{(0, 1)} \eta(Az_\star, y_\star) = \tfrac{1}{p}\left(\sum_{i=1}^{4T} e_d^{(i)}\right). 
\end{equation*}
By the definition of $A$, we have $A^\top \nabla^{(1, 0)} \eta(Az_\star, y_\star) = e_d^{(1)}$. As such, we can verify that Eq.~\eqref{Func:opt} holds true. As such, we conclude that the optimal solution $x_\star = (z_\star, y_\star)$ defined in Eq.~\eqref{Sol:opt} belongs to $\ZCal \times \YCal$ and the optimal value is
\begin{eqnarray*}
f(z_\star, y_\star) & = & \tfrac{L}{2^{p+1} p!} \cdot \left(\eta(Az_\star, y_\star) - \tfrac{1}{p(p+1)} \sum_{i=2}^{4T} (y_\star^{(i)})^{p+1} - (z_\star^{(1)} - 4T + \tfrac{1}{p}) \cdot y_\star^{(1)}\right) \\ 
& = & \tfrac{L}{2^{p+1}(p+1)!} (4T-1). 
\end{eqnarray*}
This implies the desired result. 
\paragraph{Step 3.} We now proceed to investigate the dynamics of any $p^{\textnormal{th}}$-order method under Assumption~\ref{Assumption:linear-span}. For simplicity, we denote 
\begin{equation*}
\br_k^d = \{z \in \br^d: z^{(i)} = 0 \textnormal{ for all } i = k+1, k+2, \ldots, d\}, \quad \textnormal{for all } 1 \leq k \leq d-1. 
\end{equation*}
Without loss of generality, we assume that $x_0 = \textbf{0}_{2d}$ is the initial iterate. Then, we show that the iterates $\{(z_k, y_k)\}_{k \geq 0}$ generated by any $p^{\textnormal{th}}$-order method under Assumption~\ref{Assumption:linear-span} satisfy 
\begin{equation}\label{Alg:chain-like}
z_k \in \br_{2k}^d \cap \ZCal,  \quad \textnormal{for all } 1 \leq k \leq T. 
\end{equation}
It is clear that $z_k \in \ZCal$ for all $1 \leq k \leq T$ since Assumption~\ref{Assumption:linear-span} guarantees that $\{(z_k, y_k)\}_{k \geq 1} \subseteq \XCal$. Thus, it suffices to show that $z_k \in \br_{2k}^d$ for all $1 \leq k \leq T$. 

First of all, we prove that $\bar{x} = (\bar{z}, \bar{y}) \in \br_k^d \times \br_k^d$ with $1 \leq k \leq 2T-1$ implies that $\nabla^{(j)} F(\bar{x})[h]^j \in \br_{k+1}^d \times \br_{k+1}^d$ for all $0 \leq j \leq p-1$ and any vector $h \in \br^{2d}$. In particular, we have $A$ is an upper triangular matrix and $\bar{z} \in \br_k^d$. Thus, $A\bar{z} \in \br_k^d$. In addition, we have $\bar{y} \in \br_k^d$. Recall that we have derived in Eq.~\eqref{Func:gradient-form} that
\begin{equation*}
\nabla^{(m, n)} \eta(z, y)[h_1]^m [h_2]^n = \tfrac{(p-1)!}{(p-m)!} \cdot \begin{cases}
\sum_{i=1}^d (z^{(i)})^{p-m} (h_1^{(i)})^m h_2^{(i)}, & \textnormal{if } n = 1, \\
\sum_{i=1}^d (z^{(i)})^{p-m} y^{(i)} (h_1^{(i)})^m, & \textnormal{if } n = 0, \\
0, & \textnormal{otherwise}. 
\end{cases}
\end{equation*}
Putting these pieces together yields
\begin{eqnarray*}
\tfrac{\partial}{\partial h_1}(\nabla^{(m, n)} \eta(A\bar{z}, \bar{y})[Ah_1]^m [h_2]^n) & = & \sum_{i=1}^k c_i^{(m, n)} A^\top e_d^{(i)}, \\ 
\tfrac{\partial}{\partial h_2}(\nabla^{(m, n)} \eta(A\bar{z}, \bar{y})[h_1]^m [h_2]^{n-1}) & = & \sum_{i=1}^k d_i^{(m, n)} e_d^{(i)}, 
\end{eqnarray*}
where $c_i^{(m, n)}$ and $d_i^{(m, n)}$ are certain coefficients for $1 \leq m + n \leq p$ and $1 \leq i \leq k$ (these parameters are defined in an implicit form as in~\cite{Nesterov-2021-Implementable}). Thus, we have
\begin{eqnarray*}
\tfrac{\partial}{\partial h_1}(\nabla^{(m, n)} \eta(A\bar{z}, \bar{y})[Ah_1]^m [h_2]^n) & \in & \br_{k+1}^d, \\
\tfrac{\partial}{\partial h_2}(\nabla^{(m, n)} \eta(A\bar{z}, \bar{y})[h_1]^m [h_2]^{n-1}) & \in & \br_k^d \subseteq \br_{k+1}^d. 
\end{eqnarray*}
Let us compute $F(\bar{x})$ and $\nabla F(\bar{x})[h]$ explicitly. We have 
\begin{equation*}
F(\bar{x}) = \tfrac{L}{2^{p+1} p!} \cdot \begin{bmatrix}  
\tfrac{\partial}{\partial h_1}(\nabla^{(1, 0)} \eta(A\bar{z}, \bar{y})[Ah_1])  - \bar{y}^{(1)} e_d^{(1)} \\ 
- \tfrac{\partial}{\partial h_2}(\nabla^{(0, 1)} \eta(A\bar{z}, \bar{y})[h_2]) + \tfrac{1}{p} \sum_{i=2}^{2T} (\bar{y}^{(i)})^p \cdot e_d^{(i)} + (\bar{z}^{(1)} - 2T + \tfrac{1}{p})e_d^{(1)} 
\end{bmatrix}, 
\end{equation*}
and
\begin{equation*}
\nabla F(\bar{x})[h] = \tfrac{L}{2^{p+1} p!} \cdot \begin{bmatrix} 
\tfrac{\partial}{\partial h_1}(\nabla^{(2, 0)} \eta(A\bar{z}, \bar{y})[Ah_1]^2) + \tfrac{\partial}{\partial h_2}(\nabla^{(1, 1)} \eta(A\bar{z}, \bar{y})[Ah_1][h_2]) - h_2^{(1)} e_d^{(1)} \\ 
- \tfrac{\partial}{\partial h_1}(\nabla^{(1, 1)} \eta(A\bar{z}, \bar{y})[Ah_1][h_2]) + h_1^{(1)} e_d^{(1)} - \tfrac{\partial}{\partial h_2}(\nabla^{(0, 2)} \eta(A\bar{z}, \bar{y})[h_2]^2) + \sum_{i=2}^{2T} (\bar{y}^{(i)})^{p-1} h_2^{(i)} e_d^{(i)} \end{bmatrix}. 
\end{equation*}
This together with 
\begin{equation*}
\tfrac{\partial}{\partial h_1}(\nabla^{(m, n)} \eta(A\bar{z}, \bar{y})[Ah_1]^m [h_2]^n),  \ \tfrac{\partial}{\partial h_2}(\nabla^{(m, n)} \eta(A\bar{z}, \bar{y})[h_1]^m [h_2]^{n-1}) \in \br_{k+1}^d
\end{equation*}
yields $F(\bar{x}), \nabla F(\bar{x})[h] \in \br_{k+1}^d \times \br_{k+1}^d$ for any $h \in \br^{2d}$. Similarly, we have 
\begin{equation*}
\nabla^{(j)} F(\bar{x})[h]^j \in \br_{k+1}^d \times \br_{k+1}^d, \quad \textnormal{for all } 0 \leq j \leq p-1 \textnormal{ and any } h \in \br^{2d}. 
\end{equation*}
Second, we prove that $x_k = (z_k, y_k) \in \br_k^d \times \br_k^d$ with $1 \leq k \leq 2T-1$ and $\nabla^{(j)} F(\bar{x})[h]^j \in \br_{k+1}^d \times \br_{k+1}^d$ for all $0 \leq j \leq p-1$ and any vector $h \in \br^{2d}$ implies that $x_{k+1} = (z_{k+1}, y_{k+1}) \in \br_{k+1}^d \times \br_{k+1}^d$. Indeed, we recall from Assumption~\ref{Assumption:linear-span} that: for all $k \geq 0$, we have that $x_{k+1} \in \XCal$ satisfies that $\langle \Phi_{a, \gamma, \bar{x}}(x_{k+1} - x_k), x - x_{k+1} \rangle \geq 0$ for all $x \in \XCal$, where $\Phi_{a, \gamma, \bar{x}}(\cdot)$ is defined by  
\begin{equation*}
\Phi_{a, \gamma, \bar{x}}(h) = a_0 F(\bar{x}) + \sum_{i=1}^{p-1} a_i \nabla^{(i)} F(\bar{x})[h]^i + \gamma\|h\|^{p-1}h. 
\end{equation*}
Since $\nabla^{(j)} F(\bar{x})[h]^j \in \br_{k+1}^d \times \br_{k+1}^d$ for all $0 \leq j \leq p-1$ and any vector $h \in \br^{2d}$, we have $g := a_0 F(\bar{x}) + \sum_{i=1}^{p-1} a_i \nabla^{(i)} F(\bar{x})[h]^i \in \br_{k+1}^d \times \br_{k+1}^d$. This implies that $x_{k+1} \in \XCal$ satisfies that 
\begin{equation}\label{Alg:subprob}
\langle g + \gamma\|x_{k+1} - x_k\|^{p-1}(x_{k+1} - x_k), x - x_{k+1} \rangle \geq 0 \textnormal{ for all } x \in \XCal, 
\end{equation}
where $x_k \in \br_k^d \times \br_k^d$, $g \in \br_{k+1}^d \times \br_{k+1}^d$ and $\gamma > 0$. Here, $\XCal = \ZCal \times \YCal$ where
\begin{align*}
& \ZCal = \left\{z \in \br^d: 
\begin{array}{c}
0 \leq z^{(i)} \leq 4T-i+1 \textnormal{ and } z^{(i+1)} \leq z^{(i)} \textnormal{ for all } 1 \leq i \leq 4T \\ 
\textnormal{ and } z^{(i)} = 0 \textnormal{ for all } i > 4T 
\end{array}\right\}, \\
& \YCal = \{y \in \br^d: 0 \leq y^{(i)} \leq 1 \textnormal{ for all } 1 \leq i \leq 4T \textnormal{ and } y^{(i)} = 0 \textnormal{ for all } i > 4T\}. 
\end{align*}
We claim that $x_{k+1} = (z_{k+1}, y_{k+1}) \in \br_{k+1}^d \times \br_{k+1}^d$ and consider using a proof by contradiction. Indeed, by definition, we have $x_{k+1} \in \XCal = \ZCal \times \YCal$. Thus, if $x_{k+1} = (z_{k+1}, y_{k+1}) \notin \br_{k+1}^d \times \br_{k+1}^d$, we have that either $z_{k+1} \notin \br_{k+1}^d$ or $y_{k+1} \notin \br_{k+1}^d$. We study these two cases separately as follows: 

For the former case, we have $z_{k+1}^{(i+1)} \leq z_{k+1}^{(i)}$ for all $1 \leq i \leq 2T$. This together with $z_{k+1} \notin \br_{k+1}^d$ implies that there must exist $j \geq k+2$ such that $z_{k+1}^{(j)} > 0$ and $z_{k+1}^{(j+1)} = 0$. Since $x_k \in \br_k^d \times \br_k^d$ and $g \in \br_{k+1}^d \times \br_{k+1}^d$, the $j^\textnormal{th}$ element of $g + \gamma\|x_{k+1} - x_k\|^{p-1}(x_{k+1} - x_k)$ is strictly positive. We can let $x = x_{k+1}$ except for the $j^\textnormal{th}$ element being 0. Then, it is clear that $x \in \XCal$ and 
\begin{equation*}
\langle g + \gamma\|x_{k+1} - x_k\|^{p-1}(x_{k+1} - x_k), x - x_{k+1} \rangle < 0, 
\end{equation*}
which violates Eq.~\eqref{Alg:subprob} and thus contradicts the definition of $x_{k+1}$. The similar argument is valid for the latter case where $y_{k+1} \notin \br_{k+1}^d$ is assumed and leads to the same contradiction. Putting these pieces yields the desired result. 

Finally, we prove that $x_k \in \br_{2k}^d \times \br_{2k}^d$ for all $0 \leq k \leq T$ using an inductive argument. For the case of $k = 0$, we have $x_0 = \textbf{0}_{2d} \in \br_0^d \times \br_0^d$. For the case of $k=1$, we have $F(x_0) \in \br_1^d \times \br_1^d$ which further implies that $s \in \textnormal{Lin}(F(x_0)) \subseteq \br_1^d \times \br_1^d$ and $\bar{x} = \proj_\XCal(x_0 + s) \in \br_1^d \times \br_1^d$. Then, we have $\nabla^{(j)} F(\bar{x})[h]^j \in \br_2^d \times \br_2^d$ for all $0 \leq j \leq p-1$ and any vector $h \in \br^{2d}$. This together with $x_0 \in \br_1^d \times \br_1^d$ guarantees that $x_1 \in \br_2^d \times \br_2^d$. The same argument holds for the case of $k=2$. Indeed, $F(x_0) \in \br_1^d \times \br_1^d$ and $F(x_1) \in \br_3^d \times \br_3^d$. Thus, we have $s \in \textnormal{Lin}(F(x_0), F(x_1)) \subseteq \br_3^d \times \br_3^d$ and $\bar{x} \in \br_3^d \times \br_3^d$. This implies that $\nabla^{(j)} F(\bar{x})[h]^j \in \br_4^d \times \br_4^d$ for all $0 \leq j \leq p-1$ and any vector $h \in \br^{2d}$. Since $x_1 \in \br_3^d \times \br_3^d$, we have $x_2 \in \br_4^d \times \br_4^d$. Repeating these arguments yields that $x_k \in \br_{2k}^d \times \br_{2k}^d$ for all $1 \leq k \leq T$. Thus, $z_k \in \br_{2k}^d \cap \ZCal$ for all $1 \leq k \leq T$. 

\paragraph{Step 4.} We now compute a lower bound on $\max_{y \in \YCal} f(z_k, y)$. Eq.~\eqref{Alg:chain-like} in \textbf{Step 3} implies that 
\begin{equation*}
\max_{y \in \YCal} f(z_k, y) \geq \min_{z \in \br_{2k}^d \cap \ZCal} \max_{y \in \YCal} f(z, y) \geq \min_{z \in \br_{2T}^d \cap \ZCal} \max_{y \in \YCal} f(z, y), \quad \textnormal{for all } 0 \leq k \leq T. 
\end{equation*}
We claim that $\min_{z \in \br_{2T}^d \cap \ZCal} \max_{y \in \YCal} f(z_k, y) \geq \frac{L}{2^{p+1} p!}(2T + \frac{2T-1}{p+1})$. We let $\WCal = \{w \in \br^d: w^{(i)} \geq 0 \textnormal{ for all } 1 \leq i \leq 4T\textnormal{ and } w^{(i)} = 0 \textnormal{ for all } i > 4T\}$ and derive that
\begin{eqnarray*}
& & \min_{z \in \br_{2T}^d \cap \ZCal} \max_{y \in \YCal} \ f(z, y) \\
& \overset{w = Az}{\geq} & \min_{w \in \br_{2T}^d \cap \WCal} \max_{y \in \YCal} \ \tfrac{L}{2^{p+1} p!}\left(\eta(w, y) - \tfrac{1}{p(p+1)} \sum_{i=2}^{4T} (y^{(i)})^{p+1} - \left(\sum_{i=1}^d w^{(i)} - 4T + \tfrac{1}{p}\right) \cdot y^{(1)}\right) \\
& = & \min_{w \in \br_{2T}^d \cap \WCal} \max_{y \in \YCal} \ \tfrac{L}{2^{p+1} p!}\left(\tfrac{1}{p} \sum_{i=1}^d (w^{(i)})^p \cdot y^{(i)} - \tfrac{1}{p(p+1)} \sum_{i=2}^{4T} (y^{(i)})^{p+1} - \left(\sum_{i=1}^d w^{(i)} - 4T + \tfrac{1}{p}\right) \cdot y^{(1)}\right).
\end{eqnarray*}
We see from $w \in \br_{2T}^d \cap \WCal$ that $w^{(i)} \geq 0$ for all $1 \leq i \leq 2T$ and $w^{(i)} = 0$ for all $2T+1 \leq i \leq 4T$. Fixing $w \in \br_{2T}^d \cap \WCal$, we have
\begin{eqnarray*}
& & \max_{y \in \YCal} \ \left(\tfrac{1}{p} \sum_{i=1}^d (w^{(i)})^p \cdot y^{(i)} - \tfrac{1}{p(p+1)} \sum_{i=2}^{4T} (y^{(i)})^{p+1} - \left(\sum_{i=1}^d w^{(i)} - 4T + \tfrac{1}{p}\right) \cdot y^{(1)}\right) \\
& = & \max\left\{\tfrac{1}{p}(w^{(1)})^p - \left(\sum_{i=1}^{2T} w^{(i)} - 4T + \tfrac{1}{p}\right), 0\right\} + \tfrac{1}{p} \sum_{i=2}^{2T} (w^{(i)})^p \cdot \min\{w^{(i)}, 1\} - \tfrac{1}{p(p+1)} \sum_{i=2}^{2T} (\min\{w^{(i)}, 1\})^{p+1}.
\end{eqnarray*}
The key observation is that the second and third terms are independent of $w^{(1)}$ on the right-hand side. We also have
\begin{eqnarray*}
\lefteqn{\min_{w^{(1)} \geq 0} \max\left\{\tfrac{1}{p}(w^{(1)})^p - \left(\sum_{i=1}^{2T} w^{(i)} - 4T + \tfrac{1}{p}\right), 0\right\}} \\
& = & \min_{w^{(1)} \geq 0} \max\left\{\tfrac{1}{p}(w^{(1)})^p - w^{(1)} - \sum_{i=2}^{2T} w^{(i)} + 4T - \tfrac{1}{p}, 0\right\} \ \geq \ \max\left\{4T - 1 - \sum_{i=2}^{2T} w^{(i)}, 0\right\}. 
\end{eqnarray*}
For simplicity, we define the function $g(w)$ as follows, 
\begin{equation*}
g(w) = \max\left\{4T - 1 - \sum_{i=2}^{2T} w^{(i)}, 0\right\} + \tfrac{1}{p} \sum_{i=2}^{2T} (w^{(i)})^p \cdot \min\{w^{(i)}, 1\} - \tfrac{1}{p(p+1)} \sum_{i=2}^{2T} (\min\{w^{(i)}, 1\})^{p+1}. 
\end{equation*}
Since the function $g(\cdot)$ is convex and symmetric, we have that $\min_{w \in \br_{2T}^d \cap \WCal} g(w)$ is achieved by the point with the same value of $w^{(i)}$ for all $2 \leq i \leq 2T$. Then, it suffices to solve the following one-dimensional optimization problem: 
\begin{equation*}
\min_{\eta \geq 0} \ h(\eta) = \max\{4T - 1 - \eta(2T - 1), 0\} + \tfrac{2T-1}{p} \eta^p \cdot \min\{\eta, 1\} - \tfrac{2T-1}{p(p+1)} (\min\{\eta, 1\})^{p+1}. 
\end{equation*}
For the case of $0 \leq \eta \leq 1$, we have 
\begin{equation*}
h(\eta) = 4T - 1 - \eta(2T - 1) + \tfrac{2T-1}{p+1} \eta^{p+1} \geq 2T + \tfrac{2T-1}{p+1}. 
\end{equation*}
For the case of $1 \leq \eta \leq \frac{4T-1}{2T-1}$, we have 
\begin{equation*}
h(\eta) = 4T - 1 - \eta(2T - 1) + \tfrac{2T-1}{p} \eta^p - \tfrac{2T-1}{p(p+1)} \geq 2T + \tfrac{2T-1}{p+1}. 
\end{equation*}
For the case of $\eta \geq \frac{4T-1}{2T-1}$, we have $\eta^p \geq 1 + p(\eta-1) \geq 1 + \frac{2Tp}{2T-1}$. Then, we have
\begin{equation*}
h(\eta) = \tfrac{2T-1}{p} \eta^p - \tfrac{2T-1}{p(p+1)} \geq \tfrac{2T-1}{p}\left(1 + \tfrac{2Tp}{2T-1}\right) - \tfrac{2T-1}{p(p+1)} \geq 2T + \tfrac{2T-1}{p+1}.  
\end{equation*}
Putting these pieces together yields that 
\begin{equation*}
\min_{z \in \br_{2T}^d \cap \ZCal} \max_{y \in \YCal} \ f(z, y) \geq \tfrac{L}{2^{p+1} p!}\left(2T + \tfrac{2T-1}{p+1}\right),
\end{equation*}
which implies the desired result. 

\paragraph{Final Step.} Since the point $(z_\star, y_\star) \in \ZCal \times \YCal$ is an optimal saddle-point solution, we have 
\begin{equation*}
\max_{y \in \YCal} f(z_k, y) - \min_{z \in \ZCal} f(z, y_k) \geq \max_{y \in \YCal} f(z_k, y) - f(z_\star, y_\star). 
\end{equation*}
Combining the results from \textbf{Step 2} and \textbf{Step 4}, we have
\begin{equation*}
\min_{0 \leq k \leq T}\left\{\max_{y \in \YCal} f(z_k, y) - \min_{z \in \ZCal} f(z, y_k)\right\} \geq \tfrac{L}{2^{p+1} p!}(2T - \tfrac{2T}{p+1}) \overset{p \geq 2}{\geq} \tfrac{2TL}{2^p(p+1)!}. 
\end{equation*} 
Note that we set $D_\ZCal = 8T^{3/2}$ and $D_\YCal = 2\sqrt{T}$ (cf. the definition of $\ZCal$ and $\YCal$) and have $D_\ZCal D_\YCal^p = 2^{p+3}T^{(p+3)/2}$. Then, we have
\begin{equation*}
\min_{0 \leq k \leq T}\left\{\max_{y \in \YCal} f(z_k, y) - \min_{z \in \ZCal} f(z, y_k)\right\} \geq \left(\tfrac{1}{4^{p+2}(p+1)!}\right)LD_\ZCal D_\YCal^p T^{-\frac{p+1}{2}}. 
\end{equation*}
This completes the proof. 

\section{Conclusions}\label{sec:conclu}
We have proposed and analyzed a new $p^{\textnormal{th}}$-order method---\textsf{Perseus}---for finding a weak solution of smooth and monotone variational inequalities (VIs) when $F$ is $(p-1)^{\textnormal{th}}$-order $L$-smooth. All of our theoretical results are based on the standard assumption that the subproblem arising from a $(p-1)^{\textnormal{th}}$-order Taylor expansion of $F$ can be computed approximately in an efficient manner. For the case of $p \geq 2$, the best existing $p^{\textnormal{th}}$-order methods can achieve a global rate of $O(\epsilon^{-2/(p+1)}\log\log(1/\epsilon))$~\citep{Bullins-2022-Higher,Lin-2023-Monotone,Jiang-2022-Generalized} but require a nontrivial line-search procedure at each iteration.  The open question has been whether it is possible to design a simple and optimal high-order method that achieves a global rate of $O(\epsilon^{-2/(p+1)})$ while dispensing with line search.

Our results settle this open problem. In particular, our $p^{\textnormal{th}}$-order method converges to a weak solution with a global rate of $O(\epsilon^{-2/(p+1)})$. A lower bound is proved in the monotone setting under a generalized linear span assumption, showing that our method is \textit{optimal} in the monotone setting. The restarted versions attain a global linear rate for $p^{\textnormal{th}}$-order uniformly monotone VIs and a local superlinear rate for strongly monotone VIs. Moreover, we prove a global rate of $O(\epsilon^{-2/p})$ for solving smooth and nonmonotone VIs satisfying the Minty condition and extend these results under $p^{\textnormal{th}}$-order uniform Minty and strong Minty conditions. Future research include the investigation of lower bounds for the nonmonotone setting with the Minty condition and the comparative study of various lower-order methods in high-order smooth VI problems; see~\citet{Nesterov-2021-Auxiliary,Nesterov-2021-Inexact,Nesterov-2021-Superfast} for recent examples of such comparisons in convex optimization. 

\section*{Acknowledgments}
This work was supported in part by the Mathematical Data Science program of the Office of Naval Research under grant number N00014-18-1-2764 and by the Vannevar Bush Faculty Fellowship program
under grant number N00014-21-1-2941.

\bibliographystyle{plainnat}
\bibliography{ref}

\end{document}